\documentclass[11pt]{amsart}
\usepackage[foot]{amsaddr}

\usepackage[utf8]{inputenc}

\usepackage[unicode=true,pdfusetitle,
 bookmarks=true,bookmarksnumbered=false,bookmarksopen=false,
 breaklinks=false,pdfborder={0 0 0},pdfborderstyle={},backref=false,colorlinks=false]
 {hyperref}

\usepackage{graphicx}
\usepackage{xcolor}

\usepackage{geometry}
\usepackage{amssymb,amsmath}
\usepackage{amsfonts}

\numberwithin{equation}{section}
\numberwithin{figure}{section}

\theoremstyle{plain}
\newtheorem{thm}{\protect\theoremname}[section]
\theoremstyle{definition}
\newtheorem{defn}[thm]{\protect\definitionname}
\theoremstyle{plain}
\newtheorem{lem}[thm]{\protect\lemmaname}
\theoremstyle{remark}
\newtheorem{rem}[thm]{\protect\remarkname}
\theoremstyle{plain}

\theoremstyle{definition}

\theoremstyle{remark}
\newtheorem*{rem*}{\protect\remarkname}
\theoremstyle{definition}
\newtheorem{example}[thm]{\protect\examplename}
\theoremstyle{plain}
\newtheorem{cor}[thm]{\protect\corollaryname}

\providecommand{\definitionname}{Definition}
\providecommand{\lemmaname}{Lemma}
\providecommand{\remarkname}{Remark}
\providecommand{\theoremname}{Theorem}
\providecommand{\propositionname}{Proposition}
\providecommand{\questionname}{Question}
\providecommand{\examplename}{Example}
\providecommand{\corollaryname}{Corollary}

\title{Multivariable pseudospectrum in $C^*$-algebras}

\author{Alexander Cerjan}
\email{awcerja@sandia.gov}
\address{Center for Integrated Nanotechnologies,
Sandia National Laboratories,
Albuquerque, New Mexico 87185, USA}

\author{Vasile Lauric}
\email{vasile.lauric@famu.edu}
\address{Department of Mathematics, 316 Jackson-Davis Hall, 1617 S Martin Luther King Jr. Blvd,
Tallahassee, Florida 32307, USA}

\author{Terry A. Loring }
\email{loring@math.unm.edu}
\address{Department of Mathematics and Statistics,
University of New Mexico,
Albuquerque, New Mexico, 87123, USA}

\date{}



\begin{document}

\maketitle
\thispagestyle{empty}

\begin{abstract}
We look at various forms of spectrum and  associated pseudospectrum that can be defined for noncommuting $d$-tuples of Hermitian elements of a $C^*$-algebra.  The emphasis is on theoretical calculations of examples, in particular for noncommuting pairs and triple of operators on infinite dimensional Hilbert space. In particular, we look at the universal pair of projections in a $C^*$-algebra, the usual position and momentum operators, and triples of tridiagonal operators. 
\end{abstract}

\section{Introduction}

Given  $\mathbb{A}=(A_{1},\dots,A_{d})$,
a noncommuting $d$-tuple of bounded linear Hermitian operators on
Hilbert space, so $A_{j}=A_{j}^{*}\in\mathcal{B}(\mathcal{H})$, there
are many competing notions of a joint spectrum. The one that has been
involved in recent developments in many areas of physics \cite{cerjan2023spectral,cerjan_operator_Maxwell_2022,chadha_prb_2024,Cheng2023,dixon_classifying_2023,fulga_aperiodic_2016,liu_prb_2023,Ochkan2024} is the Clifford
spectrum $\Lambda(\mathbb{A})$. This is a closed, bounded subset
of $\mathbb{R}^{d}$. To define it, we define first the \emph{spectral
localizer }
\begin{equation*}
L_{\boldsymbol{\lambda}}(\mathbb{A})=\sum_{j=1}^{d}(A_{j}-\lambda_{j}I)\otimes\Gamma_{j}.
\end{equation*}
Here $(\Gamma_{1},\dots,\Gamma_{d})$ are matrices that satisfy the
relations
\begin{align*}
\Gamma_{j}^{*} & =\Gamma_{j},\quad(j=1,\dots,d)\\
\Gamma_{j}^{2} & =1,\quad(j=1,\dots,d)\\
\Gamma_{j}\Gamma_{k} & =-\Gamma_{k}\Gamma_{j},\quad(j\neq k)
\end{align*}
that we call a \emph{representation of the Clifford relations}. 

It does not matter what representation we use \cite[Lemma~1.2]{CerjanLoring2023EvenSpheres},
so generally we take an irreducible representation, which means these matrices 
are of size $2^{\left\lfloor d/2\right\rfloor }$. For $d=2$ a standard
choice is 
\begin{equation}
\label{eq:two_gamma_matrices}
\Gamma_{1}=\sigma_{x},\,\Gamma_{2}=\sigma_{y}
\end{equation}
and for $d=3$ standard choice is 
\begin{equation}
\label{eq:three_gamma_matrices}
\Gamma_{1}=\sigma_{x},\,\Gamma_{2}=\sigma_{y},\,\Gamma_{3}=\sigma_{z}.
\end{equation}
Here we are using the Pauli spin matrices
\begin{equation*}
\sigma_{x}=\begin{bmatrix}0 & 1\\
1 & 0
\end{bmatrix},\,\sigma_{y}=\begin{bmatrix}0 & -i\\
i & 0
\end{bmatrix},\,\sigma_{z}=\begin{bmatrix}1 & 0\\
0 & -1
\end{bmatrix}.
\end{equation*}
The \emph{Clifford spectrum} is then 
\begin{equation}
\Lambda(\mathbb{A})=\left\{ \left.\boldsymbol{\lambda}\in\mathbb{R}^{d}\right|L_{\boldsymbol{\lambda}}(\mathbb{A})\text{ is not invertible}\right\} .\label{eq:Def_of_Clfford_spectrum}
\end{equation}
We will see that an interesting subset of this is the \emph{essential
Clifford spectrum}, defined as 
\begin{equation*}
\Lambda^{e}(\mathbb{A})=\left\{ \left.\boldsymbol{\lambda}\in\mathbb{R}^{d}\right|L_{\boldsymbol{\lambda}}(\mathbb{A})\text{ is not Fredholm}\right\} .
\end{equation*}

From a mathematical point of view, there is much that is unknown about the
Clifford spectrum. We are not even sure if it can be the empty set \cite[\S8]{CerjanLoring2023EvenSpheres},
although a lot of evidence says this does not happen. Most of the examples
examined so far have required computer calculations, so most examples
have been finite-dimensional. Here we move on to infinite-dimensional
examples. In most cases, we work inside the Toeplitz algebra, so that
we work with easy $C^{*}$-algebras inside the Calkin algebra to calculate
most of the Clifford spectrum, and then reduce the search for the
rest of the Clifford spectrum to a calculation involving some difference equations.
We are then able to find infinite-dimensional examples where the Clifford spectrum looks
very different from the Clifford spectra that arise from similar examples  finite-dimensions.

If $\mathbb{A}=(A_{1},\dots,A_{d})$, where each  $A_{j}$ is an Hermitian element
of a unital $C^{*}$-algebra $\mathcal{A}$, the definition of Clifford
spectrum still makes sense. Taking the $\Gamma_{j}$ to be matrices in $M_{2^{\left\lfloor d/2\right\rfloor }}(\mathbb{C})$, we treat
$L_{\boldsymbol{\lambda}}(\mathbb{A})$ as an element of $M_{2^{\left\lfloor d/2\right\rfloor }}(\mathcal{A})\cong\mathcal{A}\otimes M_{2^{\left\lfloor d/2\right\rfloor }}(\mathbb{C})$.  We now are treating $I$ as the identity element of $\mathcal{A}$.
As $\Lambda(\mathbb{A})$ is defined in Eq.~\ref{eq:Def_of_Clfford_spectrum},
in term of invertiblity, we find we have spectral permanence. That
is, if $\mathcal{A}$ is a unital $C^{*}$-subalgebra of $\mathcal{B},$
then we get the same result if we compute $\Lambda(\mathbb{A})$ working
in $M_{2^{\left\lfloor d/2\right\rfloor }}(\mathcal{B})$ as we do
if working in $M_{2^{\left\lfloor d/2\right\rfloor }}(\mathcal{A})$.

\section{Symmetries and $*$-homomorphisms}

We collect here some basic lemmas on the Clifford spectrum in $C^*$-algebras.  We already discussed spectral permanence, so we already know the behavior of the Clifford spectrum with respect to an embedding of $C^*$-algebras.

\begin{thm}
Suppose that $\varphi:\mathcal{A}\rightarrow\mathcal{B}$ is a unital $*$-homomorphism between unital $C^{*}$-algebras.
If $A_{1},\dots,A_{d}$ are Hermitian elements of $\mathcal{A}$
then
\begin{equation}
    \label{eqn:Clifford_gets_smaller}
    \Lambda(\varphi(A_{1}),\dots,\varphi(A_{d}))
  \subseteq   \Lambda(A_{1},\dots,A_{d}).
\end{equation}
If $\varphi$ is one-to-one, then the inclusion in Equation~\textup{\ref{eqn:Clifford_gets_smaller}} becomes an equality.
\end{thm}

\begin{proof}
Since
\begin{equation*}
\left(\varphi\otimes I\right)\left(L_{\boldsymbol{\lambda}}\left(A_{1},\dots,A_{d}\right)\right) =
L_{\boldsymbol{\lambda}}\left(\varphi(A_{1}),\dots,\varphi(A_{d})\right),
\end{equation*}
we know that if $L_{\boldsymbol{\lambda}}\left(A_{1},\dots,A_{d}\right)$
is invertible then $L_{\boldsymbol{\lambda}}\left(\varphi(A_{1}),\dots,\varphi(A_{d})\right)$
is also invertible. 
\end{proof}

\begin{cor} \label{cor:unitary_invariance}
If $A_{1},\dots,A_{d}$ in $\mathcal{B}(\mathcal{H}$) are Hermitian
and $U$ is a unitary operator on $\mathcal{H}$ then 
\begin{equation*}
\Lambda(UA_{1}U^{*},\dots,UA_{d}U^{*})=\Lambda(A_{1},\dots,A_{d})
\end{equation*}
and $ $
\begin{equation*}
\Lambda^{\textup{e}}(UA_{1}U^{*},\dots,UA_{d}U^{*})=\Lambda^{\textup{e}}(A_{1},\dots,A_{d}).
\end{equation*}
\end{cor}

For many purposes, such as proving that a symmetry in $\mathbb{A}$ leads to a symmetry in $\Lambda(\mathbb{A})$, we need to know that we have complete flexibility in selecting the $\Gamma_j$.

\begin{lem}
Suppose $\Gamma_{1},\dots,\Gamma_{d}$ form a representation of the Clifford relations in $\boldsymbol{M}_{r}(\mathbb{C})$ and $\Gamma'_{1},\dots,\Gamma'_{d}$  form a representation of the Clifford relations in $\boldsymbol{M}_{s}(\mathbb{C})$.
If $A_{1},\dots,A_{d}$ are Hermitian elements of unital $C^{*}$-algebra $\mathcal{A}$ then
\begin{equation*}
\sum_{j=1}^{d}\left(A_{j}-\lambda_{j}\right)\otimes\Gamma_{j}
\end{equation*}
is invertible if, and only if,
\begin{equation*}
\sum_{j=1}^{d}\left(A_{j}-\lambda_{j}\right)\otimes \Gamma'_{j}
\end{equation*}
is invertible.
\end{lem}

The proof of this lemma is essentially the same as the proof of \cite[Lemma~1.2]{CerjanLoring2023EvenSpheres} and is omitted.  We need also the following lemmas from \cite{CerjanLoring2023EvenSpheres}, also generalized to the $C^*$-algebra setting.  Again the proofs are almost identical to the matrix case and are omitted.  We will denote by
$O(d)$ the real-valued orthogonal matrices of size $d.$

\begin{lem}
\label{lem:Unitary_action_on_Clifford} Suppose $\left(A_{1},\dots,A_{d}\right)$
is a $d$-tuple of Hermitian elements of unital $C^*$-algebra $\mathcal{A}$ 
and that  $U\in O(d)$. Suppose $\boldsymbol{\mathbb{\lambda}}\in\mathbb{R}^{d}$.
The $d$ elements
\begin{equation*}
\hat{A}_{j}=\sum_{s}u_{js}A_{s}
\end{equation*}
are also Hermitian and 
\begin{equation*}
\boldsymbol{\mathbb{\lambda}}\in\Lambda\left(A_{1},\dots,A_{d}\right)\iff U\boldsymbol{\mathbb{\lambda}}\in\Lambda(\hat{A}_{1},\dots,\hat{A}_{d}).
\end{equation*}
\end{lem}

\begin{thm}
\label{thm:symmetry_in_spectrum}
Suppose $\left(A_{1},\dots,A_{d}\right)$ are  Hermitian elements in the unital $C^*$-algebra $\mathcal{A}$ and  that $U\in O(d)$.
Let 
\begin{equation*}
\hat{A}_{j}=\sum_{s}u_{js}A_{s}.
\end{equation*}
If there exists a unitary $Q$ in $\mathcal{A}$ such that 
$Q\hat{A_{j}}Q^{*}=A_{j}$ for all $j$ then 
\begin{equation*}
\boldsymbol{\mathbb{\lambda}}\in\Lambda\left(A_{1},\dots,A_{d}\right)\iff U\boldsymbol{\mathbb{\lambda}}\in\Lambda\left(A_{1},\dots,A_{d}\right).
\end{equation*}
\end{thm}

\section{The commuting and essentially commuting cases}

We know that a single operator's spectrum can look very different in the infinite-dimensional case when compared with the spectrum of a finite-dimensional counterpart.  In the finite-dimensional case, any operator (or matrix) $T$ will have finite spectrum.  In contrast, in infinite dimensions we can get any nonempty closed and bounded subset of the complex plane.  This phenomenon gives us our first examples where the Clifford spectrum looks different from how things looked in finite dimensions \cite{DeBonisLorSver_joint_spectrum,sykora2016fuzzy_space_kit}.

In the case $d=2$ the Clifford spectrum is a minor variation on the ordinary spectrum.  With the standard $\Gamma$ matrices, as in Eq.~\ref{eq:two_gamma_matrices}, the spectral localizer is
\begin{equation*}
L_{(x,y)}(A_{1},A_{2})=\left[\begin{array}{cc}
0 & A_{1}-iA_{2}-(x-iy)I\\
A_{1}+iA_{2}-(x+iy)I & 0
\end{array}\right].
\end{equation*}
This tells us immediately that $(x,y)\in \Lambda(A_1,A_2)$ exactly when $x+iy$ is in the ordinary spectrum of $A_1 +iA_2$.  For example, we can have $A_1 +iA_2$ be the bilateral shift and so have an example where
\begin{equation}
\Lambda(A_1,A_2) = \mathbb{T}^1.
\end{equation}

In finite dimensions, we have conjectured \cite[\S 8]{CerjanLoring2023EvenSpheres} that one-dimensional manifolds cannot arise as the Clifford spectrum of three matrices.  In infinite dimensions, we can get the circle as the Clifford spectrum of three operators with commutative and noncommutative examples. 

\begin{lem}
If $A_{1},\dots,A_{d}$ are pair-wise commuting Hermitian elements
of a unital $C^{*}$-algebra $\mathcal{A}$ then the Clifford spectrum
of $\mathbb{A}=(A_{1},\dots,A_{d})$ equals the standard joint spectrum.
\end{lem}

\begin{proof}
Because of spectral permanence and Corollary~\ref{cor:unitary_invariance},  we can assume that $\mathcal{A}=C(X)$ for some compact
Hausdorff space and that $A_{j}=f_{j}$ for some continuous $f_{j}:X\rightarrow\mathbb{R}$.
In this case, we can work pointwise and we find 
\begin{equation*}
L_{\boldsymbol{\lambda}}(\mathbb{A})=g:X\rightarrow M_{2^{\left\lfloor d/2\right\rfloor }}(\mathbb{C})
\end{equation*}
where 
\begin{equation*}
g(x)=\sum_{j=1}^{d}(f_{j}(x)-\lambda_{j})\Gamma_{j}.
\end{equation*}
An easy calculation shows that, for scalars $\alpha_{j}$,
\begin{equation*}
\sigma \left( \sum\alpha_{j}\Gamma_{j} \right)
= \left\{ 
\pm\sqrt{\sum\alpha_{j}^{2}}
\right\}.
\end{equation*}
Thus $g$ is singular only when there is a point $x_{0}$ in $X$
such that $f_{j}(x_0)=\lambda_{j}$ for all $j$. Thus the Clifford
spectrum is just the joint spectrum.
\end{proof}

\begin{example}
\label{ex:any_space_occurs}
If $X$ is a compact nonempty subset of $\mathbb{R}^{d}$ then there
is a commutative example of $d$ Hermitian operators on separable
Hilbert space whose Clifford spectrum equals $X$.
Notice that $X$ is metrizable so $C(X)$ is separable and so can be represented on a separable
Hilbert space.
\end{example}

\begin{thm} \label{thm:partion_commuting_inclusion}
Suppose  $A_{1},\dots,A_d$ are Hermitian elements of a unital $C^{*}$-algebra $\mathcal{A}$ and that $1\leq r<d$.  If $A_{j}$ commutes with $A_{k}$ whenever $j\leq r$ and $k>r$ then
\begin{equation*}
\Lambda(A_{1},\cdots,A_{d})\subseteq\Lambda(A_{1},\cdots,A_{r})\times\Lambda(A_{r+1},\cdots,A_{d}).
\end{equation*}
\end{thm}

\begin{proof}
We always have
\begin{equation*}
\left(L_{\boldsymbol{\lambda}}(A_{1},\cdots,A_{d})\right)^{2}=\sum(A_{j}-\lambda_{j})^{2}\otimes I+\sum_{j<k}\left[A_{j},A_{k}\right]\otimes\Gamma_{j}\Gamma_{k}
\end{equation*}
but with the given assumptions many commutators vanish.  Here we obtain 
\begin{equation*}
\left(L_{\boldsymbol{\lambda}}(A_{1},\cdots,A_{d})\right)^{2}  =\sum_{j=1}^{d}(A_{j}-\lambda_{j})^{2}\otimes I+\sum_{j<k\leq r}\left[A_{j},A_{k}\right]\otimes\Gamma_{j}\Gamma_{k}+\sum_{r\leq j<k}\left[A_{j},A_{k}\right]\otimes\Gamma_{j}\Gamma_{k}
\end{equation*}
which implies
\begin{equation}
\label{eqn:localizer_square_sum}
\left(L_{\boldsymbol{\lambda}}(A_{1},\cdots,A_{d})\right)^{2}
=
\left(L_{(\lambda_{1},\dots,\lambda_{r})}(A_{1},\cdots,A_{r})\right)^{2}+\left(L_{(\lambda_{r+1},\dots,\lambda_{d})}(A_{r+1},\cdots,A_{d})\right)^{2}.
\end{equation}
If $(\lambda_{1},\dots,\lambda_{r})\notin\Lambda(A_{1},\cdots,A_{r})$
then there is a positive $a$ such that
\begin{equation*}
a\leq\left(L_{(\lambda_{1},\dots,\lambda_{r})}(A_{1},\cdots,A_{r})\right)^{2} .
\end{equation*}
 We always have
\begin{equation*}
0\leq\left(L_{(\lambda_{r+1},\dots,\lambda_{d})}(A_{r+1},\cdots,A_{d})\right)^{2}
\end{equation*}
and so 
\begin{equation*}
a\leq\left(L_{\boldsymbol{\lambda}}(A_{1},\cdots,A_{d})\right)^{2}.
\end{equation*}
Thus 
\begin{equation*}
(\lambda_{1},\dots,\lambda_{r})\notin\Lambda(A_{1},\cdots,A_{r})\implies(\lambda_{1},\dots,\lambda_{d})\notin\Lambda(A_{1},\cdots,A_{d}).
\end{equation*}
By symmetry, 
\begin{equation*}
(\lambda_{r+1},\dots,\lambda_{d})\notin\Lambda(A_{r+1},\cdots,A_{d})\implies(\lambda_{1},\dots,\lambda_{d})\notin\Lambda(A_{1},\cdots,A_{d}).
\end{equation*}
\end{proof}

Notice that the reverse inclusion in Thereom~\ref{thm:partion_commuting_inclusion} is already false in the case where
all the $A_{j}$ commute.  We do get
equality in a simple special case.

\begin{thm} 
\label{thm:last_is_scalar}
Suppose  $A_{1},\dots,A_{d-1}$ are Hermitian elements of a unital $C^{*}$-algebra $\mathcal{A}$. For any real scalar $\alpha$ we have
\begin{equation*}
\Lambda(A_{1},\cdots,A_{d-1},\alpha I)= \Lambda(A_{1},\cdots,A_{d-1})\times \{\alpha\}.
\end{equation*}
\end{thm}

\begin{proof}
Equation~\ref{eqn:localizer_square_sum} here becomes
\begin{equation*}
\left(L_{\boldsymbol{\lambda}}(A_{1},\cdots,A_{d-1},\alpha I)\right)^{2}
=
\left(L_{(\lambda_{1},\dots,\lambda_{d-1})}(A_{1},\cdots,A_{d-1})\right)^{2}+\left( \alpha - \lambda_d \right)^{2} I
\end{equation*}
and the result follows.
\end{proof}

The following provides more evidence to support two conjectures from \cite{CerjanLoring2023EvenSpheres},
 that  when the Clifford spectrum of a $d$-tuple of $n$-by-$n$ matrices is nonempty, and that when it is finite, it must have cardinality at most $n$.
 
\begin{thm}
If $A_{1}$, $A_{2}$ and $A_{3}$ are $n$-by-$n$  Hermitian matrices
and $A_{1}$ commutes with both $A_{2}$ and  $A_{3}$ then the Clifford
spectrum of $(A_{1},A_{2},A_{3})$ is a set containing at least one point and at most $n$ points.
\end{thm}

\begin{proof}
First we look at the special case where the first matrix is scalar.  Here Theorem~\ref{thm:last_is_scalar}
tells us that $\Lambda(\alpha I,A_2,A_3)$ equals $\{\alpha\}\times \sigma(A_2 +iA_3)$ which is nonempty and can  have no more than $n$ points.

In the general case, let $\alpha_1,\dots,\alpha_r$ denote the distinct eigenvalues of $A_1$.
We can conjugate all the matrices by a unitary to ensure that 
\begin{equation*}
A_{1}=\left[\begin{array}{cccc}
\alpha_{1}I\\
 & \alpha_{2}I\\
 &  & \ddots\\
 &  &  & \alpha_{r}I
\end{array}\right].
\end{equation*}
The fact that the other matrices commute with $A_{1}$ forces them
to be block diagonal,
\begin{equation*}
A_{2}=\left[\begin{array}{cccc}
B_{1}\\
 & B_{2}\\
 &  & \ddots\\
 &  &  & B_{r}
\end{array}\right],\quad A_{3}=\left[\begin{array}{cccc}
C_{1}\\
 & C_{2}\\
 &  & \ddots\\
 &  &  & C_{r}
\end{array}\right].
\end{equation*}
Therefore
\begin{equation*}
\Lambda(A_{1},A_{2},A_{3})=\bigcup_{j=1}^{r}\Lambda(\alpha_{j}I,B_{j},C_{j})
\end{equation*}
and the result follows.
\end{proof}

\section{Singular values in infinite dimensions}

The spectrum of a matrix is generally less informative when that matrix
is not normal. In essence, this is why we cannot build as wide-ranging
functional calculus in the nonnormal case as in the normal case. Some
additional information can be found in the pseudospectum of a nonnormal
matrix. The pseudospectrum \cite{Trefethen1997PS_of_operators} of a square matrix $X$ is based on look
formally at the function 
\begin{equation}
\label{eqn:traditional_pseudospectrum}
\alpha\mapsto\left\Vert \left(X-\alpha I\right)^{-1}\right\Vert ^{-1}
\end{equation}
which takes on a gradation of values for $\alpha\in\mathbb{C}$, including
$0$ by default when $X-\lambda$ is not invertible. This function
can be seen as the pseudospectrum, but traditionally one looked at
the inverse image of $[0,\epsilon)$ and called that the $\epsilon$-pseudospectrum.

In applied math, the pseudospectrum is only defined for a single, typically non-normal, square matrix. One can easily translate this into a theory applying to two Hermitian matrices by considering $X=A_1 + iA_2$.  The norm in Eq.~\ref{eqn:traditional_pseudospectrum} can be interpreted in many ways.  
Here we are only interested in the operator norm.  We will find in more convenient to compute  $\| X^{-1} \|^{-1}$  as the smallest singular value of a matrix.
Various papers have looked at the pseudospectrum of a single nonnormal operator on separable Hilbert space, including
\cite{bottcher2003toeplitz_PS,henry2017pseudospectra,JiaFeng2021Normal_op_PS} in operator theory and \cite{komis2022robustness} in physics.

We need a replacement for the smallest singular value of a non-square matrix that works for a bounded linear operator 
$T:\mathcal{H}_{1}\rightarrow\mathcal{H}_{2}. $
For now, we are content to deal with the bounded linear operators on separable Hilbert space.  In the finite-dimensional case, one characterization of the smallest singular value is the minimum value of $\|T\boldsymbol{v}\|$ as $\boldsymbol{v}$ ranges over all unit vectors.  
We take this as a definition in the infinite-dimensional case, except we use infimum,
\begin{equation*}
s_{\min}(T)=\inf_{\|\boldsymbol{v}\|=1}\|T\boldsymbol{v}\|.
\end{equation*}

In the special case of a compact operator, there is a spectral decomposition \cite{Gohberg1990ClassesOfLinea0perators} and so all singular
values are defined.  We need only the smallest, and of course the largest since we are looking at the spectral norm.

\begin{lem}
\label{lem:s_min_for_invertible}
If $T:\mathcal{H}_{1}\rightarrow\mathcal{H}_{2} $ is invertible, then
\begin{equation*}s_{\min}(T)=\|T^{-1}\|^{-1}.
\end{equation*}
\end{lem}

\begin{proof}
By homogeneity we can compute $s_{\min}(T)$ as
\begin{align*}
s_{\min}(T) & =\inf_{\boldsymbol{v}\neq0}\frac{\|T\boldsymbol{v}\|}{\|\boldsymbol{v}\|}.
\end{align*}
If $T\boldsymbol{v}=\boldsymbol{w}$ then 
\begin{equation*}
\left(\frac{\left\Vert T\boldsymbol{v}\right\Vert }{\left\Vert \boldsymbol{v}\right\Vert }\right)^{-1}=\frac{\left\Vert T^{-1}\boldsymbol{w}\right\Vert }{\left\Vert \boldsymbol{w}\right\Vert }
\end{equation*}
and the result now follows.
\end{proof}

\begin{lem}
\label{lem:s_min_for_normal}
Suppose that $T:\mathcal{H}\rightarrow\mathcal{H}$ is a bounded linear operator.
If $T$ is normal, then
\begin{equation*}s_{\min}(T)=
\min \left\{ \left. |\lambda| \, \strut \right| 
\lambda\in\sigma(T)\right\} .
\end{equation*}
\end{lem}

\begin{proof}
Suppose $\lambda$ is an  element in the spectrum of $T$ of smallest absolute value.  Since $T$ is normal, $\lambda$ is in the approximate point spectrum so there is a sequence of unit vectors  $\boldsymbol{v}_n$ such that
\begin{equation*}
 \| T\boldsymbol{v}_n-\lambda \boldsymbol{v}_n \| \leq \frac{1}{n}.
\end{equation*}
Thus
\begin{equation*}
 \| T\boldsymbol{v}_n\| \leq |\lambda| + \frac{1}{n}
\end{equation*}
and so $s_{\min}(T)\leq |\lambda|$.

Now we prove the opposite inequality.  Since $s_{\min}(T)$ and $\sigma(T)$ are invariant under conjugation
by a unitary, we can use the spectral theorem \cite{Hall2013QuantumTheoryForMathematicians} to reduce to the case where $\mathcal{H}=L^{2}(X,\mu)$ and there is some bounded measurable $f:X\rightarrow \mathbb{C}$ so that $T\xi(x)=f(x)\xi(x)$. Let $[a,b]$ denoted the smallest closed interval that contains the essential range of $|f|$. Then $a$ is the smallest element in the spectrum of $|f|$, which by the spectral mapping theorem equals the
smallest absolute value of an element of the spectrum of $f$. Since 
\begin{equation*}
\left\Vert T\xi\right\Vert ^{2}=\int|f(x)|^{2}|\xi(x)|^{2}\,dx\geq a^{2}\int|\xi(x)|^{2}\,dx=a^{2}\left\Vert \xi\right\Vert ^{2}
\end{equation*}
we know that $s_{\min}(T)\geq a$. 
\end{proof}

Lemma~\ref{lem:s_min_for_normal} allows us to extend the definition of $s_{\min}$ to normal elements of a $C^*$-algebra.  In particular, we can now extend the definition of the Clifford pseudospectrum for the $C^*$-algebra setting.

\begin{defn}
Suppose   $A_{1},\dots,A_{d}$ are Hermitian elements of a unital $C^{*}$-algebra $\mathcal{A}$.  The \emph{Clifford pseudospectrum}
for $\mathbb{A}=(A_{1},\dots,A_{d})$ is the function $\boldsymbol{\lambda}\mapsto \mu_{\boldsymbol{\lambda}}^{\textup{C}}(\mathbb{A})$ where
\begin{equation*}
    \mu_{\boldsymbol{\lambda}}^{\textup{C}}(\mathbb{A}) = s_{\min}\left( L_{\boldsymbol{\lambda}}(\mathbb{A}) \right) .
\end{equation*}
\end{defn}

\begin{lem}
\label{lem:s_min_is_Lipschitz}
Suppose  $S,T:\mathcal{H}_{1}\rightarrow\mathcal{H}_{2}$ are both bounded
linear operators, then
\begin{equation*}
\left|s_{\min}(S)-s_{\min}(T)\right|\leq\left\Vert S-T\right\Vert .
\end{equation*}
\end{lem}

\begin{proof}
Let $\epsilon>0$ be given. Then there is a unit vector $\boldsymbol{v}$
so that 
\begin{equation*}
\|T\boldsymbol{v}\|\leq s_{\min}(T)+\epsilon.
\end{equation*}
Then 
\begin{equation*}
s_{\min}(S)\leq\|S\boldsymbol{v}\|
\leq
\|T\boldsymbol{v}\|+\|(S-T)\boldsymbol{v}\|\leq s_{\min}(T)+\epsilon+\|S-T\|
\end{equation*}
 proving
\begin{equation*}
s_{\min}(S)-s_{\min}(T)\leq\|S-T\|+\epsilon.
\end{equation*}
As this is true for all $\epsilon>0$, we have shown 
\begin{equation*}
s_{\min}(S)-s_{\min}(T) \leq \|S-T\|.
\end{equation*}
We are done, by symmetry.
\end{proof}

\begin{rem}
Unlike the case in finite dimensions, $s_{\min}(T)\neq0$ does not
always imply $T$ is invertible. However, when $T^{*}=T$ it is true
that $s_{\min}(T)\neq0$ if and only if $T^{-1}$ exists.
\end{rem}

\begin{lem}
If $T:\mathcal{H}_{1}\rightarrow\mathcal{H}_{2}$ and $S:\mathcal{H}_{2}\rightarrow\mathcal{H}_{3}$
are bounded linear operators then
\begin{align*}
s_{\min}(ST) & \geq s_{\min}(S)s_{\min}(T)
\end{align*}
\end{lem}

\begin{proof}
Recall that 
\begin{align*}
s_{\min}(T) & =\inf_{\boldsymbol{v}\neq0}\frac{\|T\boldsymbol{v}\|}{\|\boldsymbol{v}\|}.
\end{align*}
If $T$ has no kernel then, for every nonzero $\boldsymbol{v}$, we have
\begin{equation*}
\frac{\|ST\boldsymbol{v}\|}{\|\boldsymbol{v}\|}=\frac{\|ST\boldsymbol{v}\|}{\|T\boldsymbol{v}\|}\frac{\|T\boldsymbol{v}\|}{\|\boldsymbol{v}\|}\geq s_{\min}(S)s_{\min}(T)
\end{equation*}
 so 
\begin{align*}
s_{\min}(ST) & \geq s_{\min}(S)s_{\min}(T).
\end{align*}
If $T$ has a kernel then $ST$ has a kernel so both $s_{\min}(ST)$
and $s_{\min}(T)$ are zero and we are done for trivial reasons.
\end{proof}

\begin{rem*}
One might think that $T$ and $T^{*}$ and
\[
\left[\begin{array}{cc}
0 & T^{*}\\
T & 0
\end{array}\right]
\]
 all have the same $s_{\min}$.  After all, this is true in the finite-dimensional case. However, if we let $T$ be the forward
shift, then  $ s_{\min}(T)=1$
since 
$\left\Vert T\boldsymbol{v}\right\Vert =\left\Vert \boldsymbol{v}\right\Vert $
for all $\boldsymbol{v}$ (that is, $T$ is an isometry). In contrast,
there is a unit vector $\boldsymbol{v}$ so that that the backwards
shift sends $\boldsymbol{v}$ so zero, so that $\left\Vert T^*\boldsymbol{v}\right\Vert =0$.
Thus also 
\[
\left[\begin{array}{cc}
0 & T^{*}\\
T & 0
\end{array}\right]\left[\begin{array}{c}
0\\
\boldsymbol{v}
\end{array}\right]=\left[\begin{array}{c}
0\\
0
\end{array}\right]
\]
so we can conclude
\[
s_{\min}\left[\begin{array}{cc}
0 & T^{*}\\
T & 0
\end{array}\right]=s_{\min}(T^{*})=0.
\]
Finally, notice that $\sqrt{T^{*}T}=I$ while $\sqrt{TT^{*}}$
has a kernel, so
\[
s_{\min}\left(|T^{*}|\right)\neq s_{\min}\left(|T|\right)
\]
in infinite dimensions. 
\end{rem*}

   For two Hilbert spaces $\mathcal{H}_i, \ i=1,2,$
   let $\mathcal{L}(\mathcal{H}_1,\mathcal{H}_2)$ denote
   the set of all bounded linear operators from 
   $\mathcal{H}_1$ to $\mathcal{H}_2.$
   Recall that for $T\in\mathcal{L}(\mathcal{H}_1,\mathcal{H}_2),$ $|T|\in\mathcal{L}(\mathcal{H}_1)$ denotes the modulus of $T$, 
   that is the unique positive semi-definite operator so that 
   $|T|^2=T^*T$ and $\ker(T)=\ker(|T|).$

\begin{lem}
    If $T:\mathcal{H}_{1}\rightarrow\mathcal{H}_{2}$ 
    is a bounded linear operator then
    \begin{equation*}
             s_{\min}(T)=s_{\min}(|T|)=\sqrt{s_{\min}(T^{*}T)} .
    \end{equation*}
\end{lem}

\begin{proof}
  Since $\|  T\boldsymbol{v}  \|=\| \, |T|\boldsymbol{v}\, \|$ for all $\boldsymbol{v}\in\mathcal{H}_1,$
  the consequence that $s_{\min}(T)=s_{\min}(|T|)$ is straightforward.  The equality of $s_{\min}(|T|)$ with $\sqrt{s_{\min}(T^{*}T)}$ follows from Lemma~\ref{lem:s_min_for_normal} and the spectral mapping
  theorem.
  
%  On other hand since $|| \, |T|x \,\|^2\le ||T^* T x||,$ %for all $x\in\mathcal{H}_1,$
%  it implies $s_{\min}(|T|)\le \sqrt{s_{\min}(T^* T)}.$
%  If $s_{\min}(T^* T)=0,$ then $s_{\min}(T)=0.$ Thus, it %remains to prove that the equality holds in the case that %$s_{\min}(T^*T)>0.$ 
%  Let us assume that 
%  $0\le\gamma:=s_{min}(|T|)< \sqrt{s_{min}(T^* %T)}:=\sqrt{\alpha}$
%  and let $n$ be an integer so that $\gamma+\frac{1}%{n}<\sqrt{\alpha}.$
%  Thus, there exists $u_n\in\mathcal{H}_1, \ ||u_n||=1$
%  so that $|| |T|u_n||\le\gamma+\frac{1}{n},$
%  or equivalently
%  $\langle T^* T u_n, u_n \rangle=
%  || |T|u_n||^2\le\gamma^2+\frac{2\gamma}{n}+\frac{1}%{n^2},$
%  and therefore
%  $$\langle (T^*T-\gamma^2) u_n, u_n \rangle\le \frac{2
%\gamma}{n}+\frac{1}{n^2},$$
%  which implies that $\gamma^2\in\sigma_{less}(|T|^2),$
%  and thus $\gamma^2\in\sigma(|T|^2).$
%  On other hand, according the the Spectral Mapping %Theorem,
%  $\sigma(\sqrt{|T|^2-\gamma^2})\subset [\sqrt{\alpha-\gamma^2},+\infty)$
%  and $\sqrt{\alpha-\gamma^2}>0,$ and thus $\sqrt{|T|^2-\gamma^2}$ must be invertible, which contradicts the above inequality.
\end{proof}

With these lemmas, we can now extend some of the result on Clifford spectrum and pseudospectrum that have
been previously worked out in the case of finite matrices \cite{cerjan_L_V_2022quadraticPS,LoringPseudospectra}.
Theorem~\ref{thm:CliffPS_Lipschitz_near_abs(lambda)}
is a generalization of Lemma 2.3 of \cite{PalYak2017InfDimOper_PS}. That result handles one general operator, or equivalently two Hermitian operators.

\begin{thm}
\label{thm:CliffPS_Lipschitz_near_abs(lambda)}
Suppose   $A_{1},\dots,A_{d}$ are Hermitian elements of a unital $C^{*}$-algebra $\mathcal{A}$.  The Clifford pseudospectrum
of $\mathbb{A}$ is Lipschitz, specifically with
\begin{equation}
\label{eqn:Lipschitz_in_Lambda}
  \left| \mu_{\boldsymbol{\lambda}}^{C}(A_{1},\dots,A_{d})-\mu_{\boldsymbol{\nu}}^{C}(A_{1},\dots,A_{d}) \right|
    \leq
    \| \boldsymbol{\lambda} - \boldsymbol{\nu} \|
\end{equation}
where the norm is taken to be the Euclidean norm on $\mathbb{R}^n$.  Moreover
\begin{equation}
\label{eqn:CliffPS_close_to_abs(lambda)}
\left|\mu_{\boldsymbol{\lambda}}^{C}(\mathbb{A})-\|\boldsymbol{\lambda}\|\right| \leq
\left \| L_{\boldsymbol{0}}(\mathbb{A}) \right \|.
\end{equation}
\end{thm}

\begin{proof}
These statements all follow from the fact that $s_{\min}(A)$ is Lipshitz in the operator norm (Lemma~\ref{lem:s_min_is_Lipschitz}).  This result is a generalization of results in Section 7 of \cite{LoringPseudospectra}, and is proven by related methods.
The essential calculations are
\begin{equation*}
\sigma\left(L_{\boldsymbol{0}}\left(\lambda_{1}I,\cdots,\lambda_{d}I\right)\right)=\left\{ \left\Vert \boldsymbol{\lambda}\right\Vert ,-\left\Vert \boldsymbol{\lambda}\right\Vert \right\} 
\end{equation*}
and its corollary
\begin{equation*}
\left\Vert L_{\boldsymbol{0}}\left(\lambda_{1}I,\cdots,\lambda_{d}I\right)\right\Vert =\left\Vert \boldsymbol{\lambda}\right\Vert .
\end{equation*}
For (\ref{eqn:Lipschitz_in_Lambda}) note that 
\begin{equation*}
   L_{\boldsymbol{\lambda}}\left(A_1,\cdots,A_d\right) -
   L_{\boldsymbol{\nu}}\left(A_1,\cdots,A_d\right)
   =
   L_{\boldsymbol{0}}\left((\nu_1-\lambda_1)I,\cdots,(\nu_{d}-\lambda_d)I\right)
\end{equation*}
which implies that
\begin{equation*}
  \left\| L_{\boldsymbol{\lambda}}\left(A_1,\cdots,A_d\right) -
   L_{\boldsymbol{\nu}}\left(A_1,\cdots,A_d\right) \right \|
   = \| \boldsymbol{\lambda} - \boldsymbol{\nu} \|.
\end{equation*}
For (\ref{eqn:CliffPS_close_to_abs(lambda)}), we know that
\begin{equation*}
    s_{\min} \left( L_{\boldsymbol{0}}\left(\lambda_{1}I,\cdots,\lambda_{d}I\right) \right) 
    = \| \boldsymbol{\lambda}\|
\end{equation*}
and
\begin{equation*}
    \left \| L_{\boldsymbol{0}}\left(\lambda_{1}I,\cdots,\lambda_{d}I\right) +
    L_{\boldsymbol{\lambda}}\left(A_1,\cdots,A_d\right)
    \right \| =     \left \|  L_{\boldsymbol{0}}\left(A_{1},\cdots,A_{d}\right) \right \| .
\end{equation*}
\end{proof}

\begin{cor}
The Clifford spectrum for  a $d$-tuple of Hermitian elements of a $C^*$-algebra is always compact.
\end{cor}

\begin{proof}
Since Lipshcitz implies continuous we get continuity from Equation~\ref{eqn:Lipschitz_in_Lambda}.  Equation~\ref{eqn:CliffPS_close_to_abs(lambda)} tells us that then
$\left \| \boldsymbol{\lambda} \| >  \| L_{\boldsymbol{0}}(\mathbb{A}) \right \|$
we cannot have $\mu_{\boldsymbol{\lambda}}^{C}(\mathbb{A})=0$.  Therefore
the Clifford spectrum of $(A_{1},\dots,A_{d})$ is a closed subset of the ball at the origin of radius 
$\left \| L_{\boldsymbol{0}}(A_1,\dots,A_d) \right \|$.
\end{proof}

\section{Other forms of multivariable pseudospectrum}

\subsection{The quadratic pseudospectrum and more}

There are many ways to define a pseudospectrum for 
$\mathbb{A}=(A_{1},\dots,A_{d})$ that are Hermitian.  If these are operators on Hilbert space $\mathcal{H}$,  Mumford \cite{mumford2023numbers_and_beyond}
suggests associating to $\boldsymbol{\lambda}$  the number
\begin{equation}
\label{eqn:mumford_PS}
\min\left\{ \left. \max\left\Vert A_{j}\boldsymbol{v}-\lambda_{j}\boldsymbol{v}\right\Vert 
\, \strut \right|\, 
\|\boldsymbol{v}\|=1\text{ and }\textup{E}_{\boldsymbol{v}}[A_{j}]=\lambda_{j}\text{ for all }j\right\} 
\end{equation}
 where $\textup{E}_{\boldsymbol{v}}[A_{j}]=\left\langle A_{j}\boldsymbol{v},\boldsymbol{v}\right\rangle $
is the expectation for the observable $A_{j}$ when the system is
in state $\boldsymbol{v}$.  This is expected \cite[\S 1]{LoringPseudospectra} to be a very difficult, but important, minimization problem related to joint measurement.  

Fortunately, there
is a slight modification that can be made to Equation~\ref{eqn:mumford_PS} that turns it into something very computable, as the minimization problem leads to a number equal to the smallest spectral value of a Hermitian matrix.  The \textit{quadratic pseudospectrum} \cite{cerjan_L_V_2022quadraticPS} is defined as
\begin{equation}
    \mu_{\boldsymbol{\lambda}}^{\textup{Q}}(A_{1},\dots,A_{d})=\min\left\{ \left.\sqrt{\sum_{j}\left\Vert A_{j}\boldsymbol{v}-\lambda_{j}\boldsymbol{v}\right\Vert ^{2}}\,\right|\,\|\boldsymbol{v}\|=1\right\} .
\end{equation}
The practical way to compute this is to use the known equality \cite{cerjan_L_V_2022quadraticPS}
\begin{equation}
    \mu_{\boldsymbol{\lambda}}^{\textup{Q}}(A_{1},\dots,A_{d})=
    \sqrt{s_{\min}\left(Q_{\boldsymbol{\lambda}}\left(\mathbb{A}\right)\right)}
\end{equation}
where
\begin{equation}
    Q_{\boldsymbol{\lambda}}\left(\mathbb{A}\right)=\sum_{j}\left(A_{j}-\lambda_{j}\right)^{2} .
\end{equation}
This can be defined also in the case where the $A_j$ are Hermitian elements in
a unital $C^*$-algebra.  

We can also select a bump function $g:\mathbb{R}\rightarrow\mathbb{R}$
with $0\leq g\leq1$ and $g(0)=1$, some manner of a windowing function.
Lin (in [19]) looks at 
\begin{equation*}
1 - \left\Vert W(\mathbb{A})\right\Vert 
\end{equation*}
where 
\begin{equation*}
W(\mathbb{A})=g\left(A_{1}-\lambda_{1}\right)g\left(A_{2}-\lambda_{2}\right)\cdots g\left(A_{d}-\lambda_{d}\right).
\end{equation*}
Notice there needs to be some choice of the order in the product.
This is related to the windowed LDOS 
\cite{loringLuWatson2021locality} that looks at the trace of
$\left(W(\mathbb{A})\right)^*W(\mathbb{A})$.  It can be difficult to compute $g\left(A_{j}-\lambda_{1}\right)$, but it should be should be possible to numerically compute this in settings
where some of the matrices are diagonal.
For theoretical calculations, this can be a practical form of
pseudospectrum \cite{lin2024almost_commuting_and_measurement}.

\subsection{Relation between the Clifford and quadratic pseudospectrum}

It is expected that all variations on multivariable operator pseudospectrum will be related somehow.  Indeed, they should all be equal in the commutative case, and be close when the $A_j$ have small commutators.  This is known in one case, in that the Clifford and quadratic are
of bounded distance apart.  That is, we know \cite{cerjan_L_V_2022quadraticPS} that
\begin{equation}
\label{eqn:closeness_of_PS}
\left |
   \biggl( \mu_{\boldsymbol{\lambda}}^{C}\left(\mathbb{A}\right) \biggr )^2
    -
    \biggl( \mu_{\boldsymbol{\lambda}}^{Q}\left(\mathbb{A}\right) \biggr )^2
    \right |
    \leq 
    \sum_{j<k} \| [A_j, A_k]  \| .
\end{equation}

It is becoming clear that one cannot make do with a single form of pseudospectrum when studying physical systems.
The Clifford pseudospectrum is related to $K$-theory and
is finding applications in photonics \cite{cerjan_operator_Maxwell_2022,CerjLorShuba2024LocalMarkersCrystalline}, acoustics \cite{Cheng2023}, aperiodic structures \cite{fulga_aperiodic_2016}, nonlinear systems \cite{wong_nonlinear_2023}, and even non-Hermitian systems \cite{cerjan2023spectral,chadha_prb_2024,dixon_classifying_2023,
liu_prb_2023,Ochkan2024}.
On the other hand, showing that points in the Clifford spectrum correspond to states approximately localized in energy and position seems to require using something like Equation~\ref{eqn:closeness_of_PS}.

\subsection{Symmetry in the quadratic pseudospectrum}

There are symmetry properties in the Clifford spectrum that must follow from symmetry properties in $\mathbb{A}$ \cite{CerjanLoring2023EvenSpheres}.  Here we prove the same implication, but for the quadratic case.  

\begin{lem}
\label{lem:Unitary_action_on_quadratic} 
Suppose $\left(A_{1},\dots,A_{d}\right)$ is a $d$-tuple of Hermitian
elements in the unital $C^*$-algebra $\mathcal{A}$ and that $U\in O(d)$.
Suppose $\boldsymbol{\mathbb{\lambda}}\in\mathbb{R}^{d}$. The $d$
element
\begin{equation*}
\hat{A}_{j}=\sum_{s}u_{js}A_{s}
\end{equation*}
are also Hermitian, and 
\begin{equation*}
\mu_{U\boldsymbol{\lambda}}^{Q}\left(\hat{A}_{1},\dots,\hat{A}_{d}\right)=\mu_{\boldsymbol{\lambda}}^{Q}\left(A_{1},\dots,A_{d}\right).
\end{equation*}
\end{lem}

\begin{proof}
The elements $\hat{A}_{j}$ can be shown to  be Hermitian, by the same
proof  \cite[\S 2]{CerjanLoring2023EvenSpheres} as for the matrix case.
We find
\begin{align*}
\sum_{j}\hat{A}_{j}^{2} & =\sum_{r}\sum_{s}\sum_{j}u_{jr}u_{js}A_{r}A_{s}\\
 & =\sum_{r}A_{r}^{2}\\
 & =\sum_{j}A_{j}^{2}.
\end{align*}
Setting
\begin{equation*}
\tilde{\lambda}_{j}=\sum_{s}u_{js}\lambda_{s}.
\end{equation*}
we next compute
\begin{align*}
\sum_{j}\tilde{\lambda}_{j}\hat{A}_{j} & 
=\sum_{r}\sum_{s}\sum_{j}u_{js}u_{jr}\lambda_{s}A_{r}\\
 & =\sum_{r}\lambda_{r}A_{r}\\
 & =\sum_{j}\lambda_{j}A_{j}
\end{align*}
Finally
\begin{align*}
\sum_{j}\tilde{\lambda}_{j}^{2} & =\sum_{r}\sum_{s}\sum_{j}u_{rj}u_{sj}\lambda_{r}\lambda_{s}\\
 & =\sum_{r}\lambda_{r}^{2}.
\end{align*}
Thus
\begin{align*}
Q_{\widetilde{\boldsymbol{\lambda}}}\left(\hat{A}_{1},\dots,\hat{A}_{d}\right) & =\sum_{j}\hat{A}_{j}^{2}-2\sum_{j}\tilde{\lambda}_{j}\hat{A}_{j}+\sum_{j}\tilde{\lambda}_{j}^{2}I\\
 & =\sum_{j}A_{j}^{2}-2\sum_{j}\lambda_{j}A_{j}+\sum_{j}\lambda_{j}I\\
 & =Q_{\boldsymbol{\lambda}}\left(A_{1},\dots,A_{d}\right).
\end{align*}
\end{proof}

If $Q$ is a unitary and we let
$
\mathbb{B}=(QA_1Q^*,\dots,QA_1Q^*)
$
then $L_{\boldsymbol{\lambda}}(\mathbb{B})$ is unitarily equivalent to $L_{\boldsymbol{\lambda}}(\mathbb{A})$  and $Q_{\boldsymbol{\lambda}}(\mathbb{B})$ is unitarily equivalent to $Q_{\boldsymbol{\lambda}}(\mathbb{A})$.  Thus we can use Lemma~\ref{lem:Unitary_action_on_Clifford} and Lemma~\ref{lem:Unitary_action_on_quadratic}
 to obtain the following theorem.

\begin{thm}
\label{thm:symmetry_in_PS}
Suppose $\left(A_{1},\dots,A_{d}\right)$ are  Hermitian elements in the unital $C^*$-algebra $\mathcal{A}$ and  that $U\in O(d)$.
Let 
\begin{equation*}
\hat{A}_{j}=\sum_{s}u_{js}A_{s}.
\end{equation*}
If there exists a unitary $Q$ in $\mathcal{A}$ such that 
$Q\hat{A_{j}}Q^{*}=A_{j}$ for all $j$ then 
\begin{equation*}
\mu^\textup{C}_{U\boldsymbol{\lambda}}\left(A_{1},\dots,A_{d}\right)
= 
\mu^\textup{C}_{\boldsymbol{\lambda}}\left(A_{1},\dots,A_{d}\right)
\end{equation*}
and
\begin{equation*}
\mu^\textup{Q}_{U\boldsymbol{\lambda}}\left(A_{1},\dots,A_{d}\right)
= 
\mu^\textup{Q}_{\boldsymbol{\lambda}}\left(A_{1},\dots,A_{d}\right) .
\end{equation*}
\end{thm}

\section{The $C^{*}$-algebra generated by two projections}

The universal unital $C^{*}$-algebra  for the relations that define
two orthogonal projections \cite{raeburn1989two_projections} is 
\begin{equation}
\label{eqn:universal_for_two_proj}
C^{*}(\mathbb{Z}_{2}*\mathbb{Z}_{2})=\left\{ f\in C([-1,1],\boldsymbol{M}_{2})\,\mid f(-1)\text{ and }f(1)\text{ are diagonal}\right\} .
\end{equation}
We will now compute both forms of pseusospectrum for this universal
pair of projections. The formulas derived will come out simpler if
we instead think in terms of a universal pair of unitary operators
of order two. We might see these as two incompatible dichotomous observables.

The generators of $C^{*}(\mathbb{Z}_{2}*\mathbb{Z}_{2})$ are $U$
and $V$ where
\begin{equation}
\label{eqn:the_two_projectinons}
U(z)=\left[\begin{array}{cc}
1 & 0\\
0 & -1
\end{array}\right],\ V(z)=\left[\begin{array}{cc}
z & \sqrt{1-z^{2}}\\
\sqrt{1-z^{2}} & -z
\end{array}\right].
\end{equation}
First we work on the representation corresponding to evaluation at
a fixed $z$. Here the Hilbert space is only two-dimensional and the
calculations are not too extensive. They are a bit tedious so we utilized
a computer algebra package here. 

\begin{figure}
\noindent \begin{centering}
\includegraphics[viewport=86bp 0bp 420bp 310bp,clip,scale=0.4]{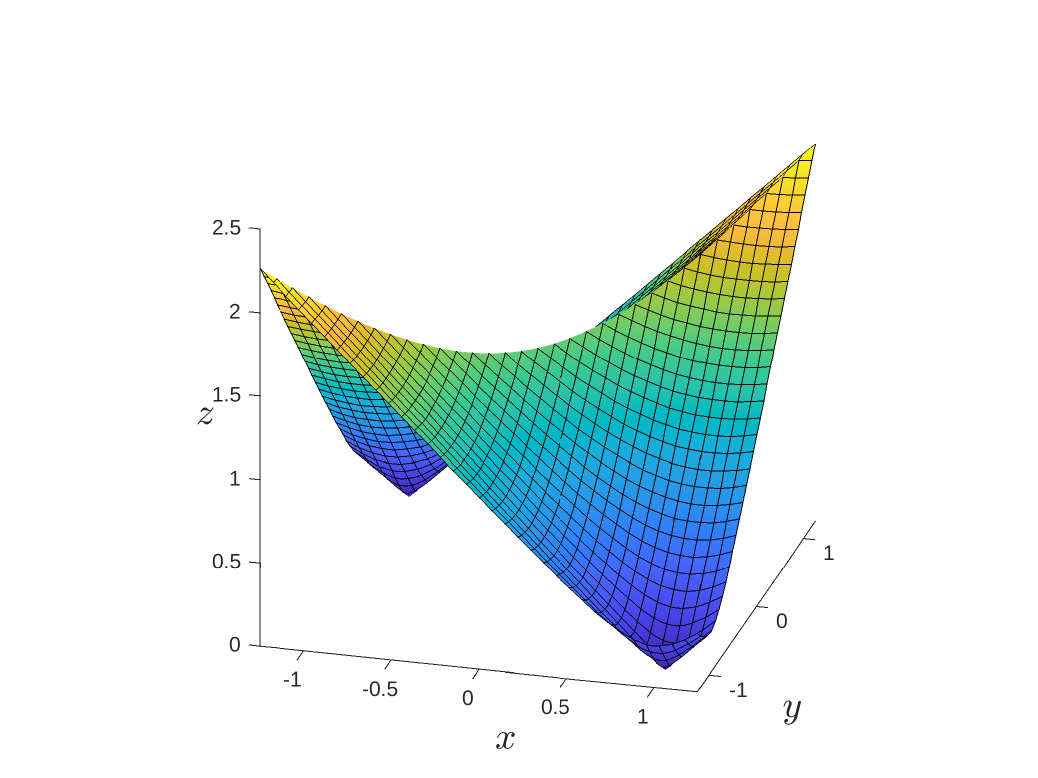}\includegraphics[viewport=86bp 0bp 420bp 310bp,clip,scale=0.4]{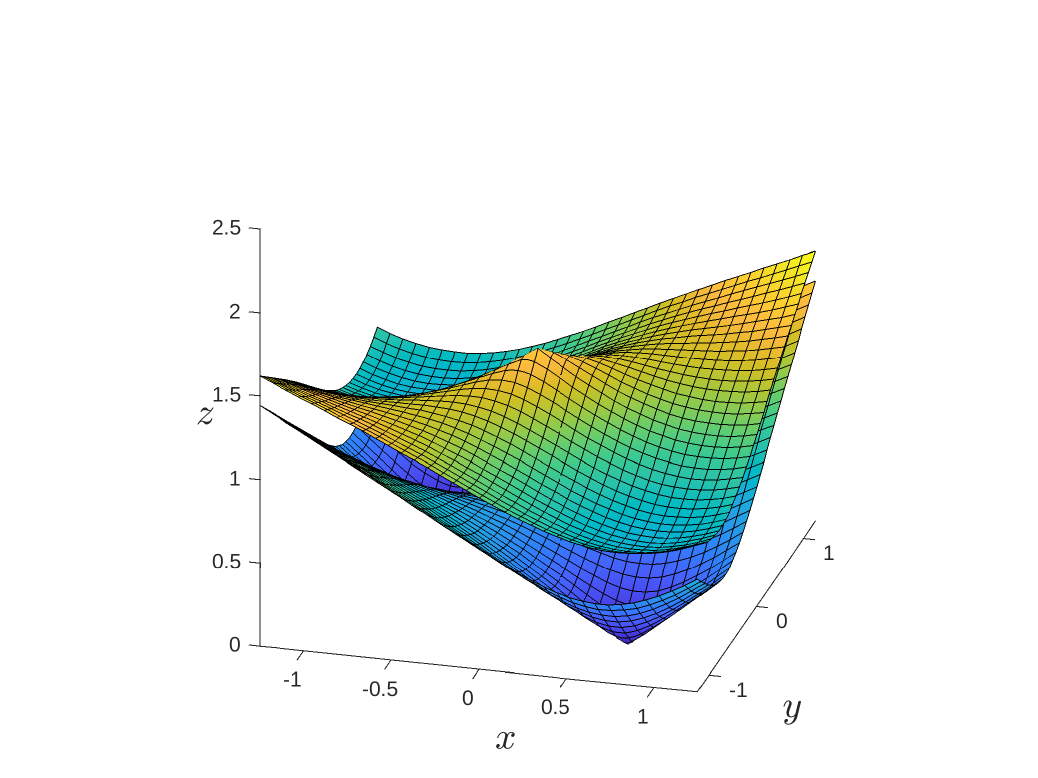}\includegraphics[viewport=86bp 0bp 420bp 310bp,clip,scale=0.4]{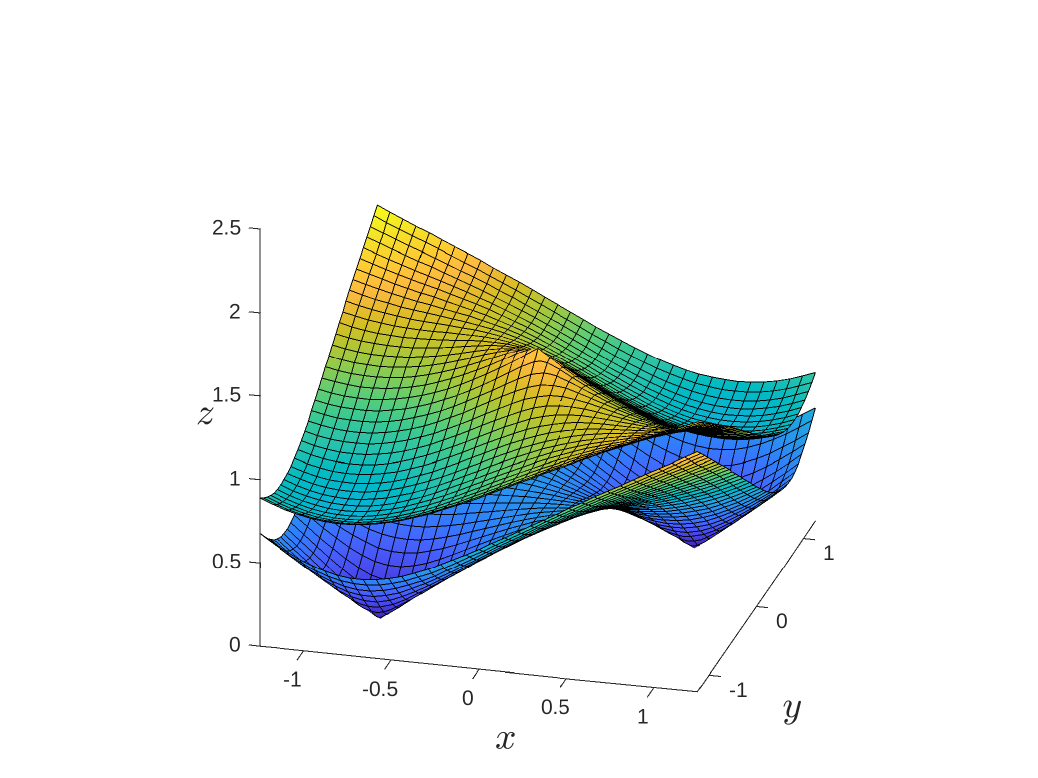}
\par\end{centering}
\noindent \begin{centering}
\includegraphics[viewport=86bp 0bp 420bp 310bp,clip,scale=0.4]{U_Vz_bothPS_z0_50}\includegraphics[viewport=86bp 0bp 420bp 310bp,clip,scale=0.4]{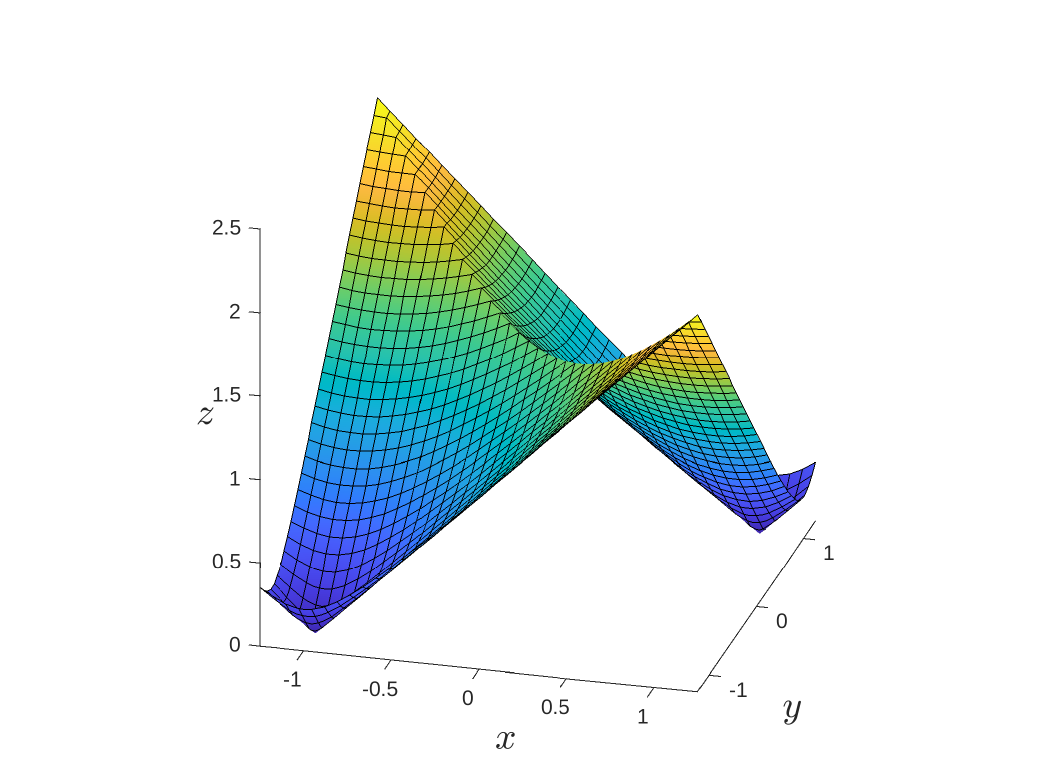}
\par\end{centering}
\caption{Plotted together the Clifford and quadratic  pseudospectrum for $(U,V_{z})$, as defined in (\ref{eqn:two-by-two-projections},) for various values of of $z$.  From the top-left to the bottom-right, the values used are $z=1$, $z=0.6$, $z=0$, $z=-0.5$, $z=-1$.}
\end{figure}

\begin{lem}
Suppose $-1\leq z\leq1$. If
\begin{equation}
\label{eqn:two-by-two-projections}
U=\left[\begin{array}{cc}
1 & 0\\
0 & -1
\end{array}\right],\ V_{z}=\left[\begin{array}{cc}
z & \sqrt{1-z^{2}}\\
\sqrt{1-z^{2}} & -z
\end{array}\right]
\end{equation}
then
\begin{equation*}
\mu_{(x,y)}^{\textup{Q}}(U,V_{z})=\sqrt{x^{2}+y^{2}+2-2\sqrt{x^{2}+2zxy+y^{2}}}
\end{equation*}
and
\begin{equation*}
\mu_{(x,y)}^{\textup{C}}(U,V_{z})=\sqrt{x^{2}+y^{2}+2-2\sqrt{x^{2}+2xyz+y^{2}+1-z^{2}}}.
\end{equation*}
\end{lem}

\begin{proof}
We compute the smallest eigenvalue of 
\begin{equation*}
Q_{(x,y)}(U,V_{z})=\left(U-xI\right)^{2}+\left(V_{z}-yI\right)^{2},
\end{equation*}  
finding that this equals 
\begin{equation*}
Q_{(x,y)}=\left[\begin{array}{cc}
x^{2}+y^{2}-2x-2zy+2 & -2y\sqrt{1-z^{2}}\\
-2y\sqrt{1-z^{2}} & x^{2}+y^{2}+2x+2zy+2
\end{array}\right].
\end{equation*}  
This has eigenvalues
\begin{equation*}
x^{2}+y^{2}+2\pm2\sqrt{x^{2}+2zxy+y^{2}}.
\end{equation*}  
The smaller of the eigenvalues is with the minus sign. To get to $\mu_{(x,y)}^{Q}(P,Q_{z})$
we apply square root. 

For the Clifford pseudospectrum, we need the eigenvalues of
\begin{equation*}
\left[\begin{array}{cc}
0 & \left(U-xI\right)-i\left(V_{z}-yI\right)\\
\left(U-xI\right)+i\left(V_{z}-yI\right) & 0
\end{array}\right]
\end{equation*}  
which are determined by the singular values of 
\begin{equation*}
\left(U-xI\right)+i\left(V_{z}-yI\right).
\end{equation*}  
These are
\begin{equation*}
\sqrt{x^{2}+y^{2}+2\pm2\sqrt{x^{2}+2xyz+y^{2}+1-z^{2}}}.
\end{equation*}  
\end{proof}

\begin{figure}
\noindent \begin{centering}
\includegraphics[scale=0.6]{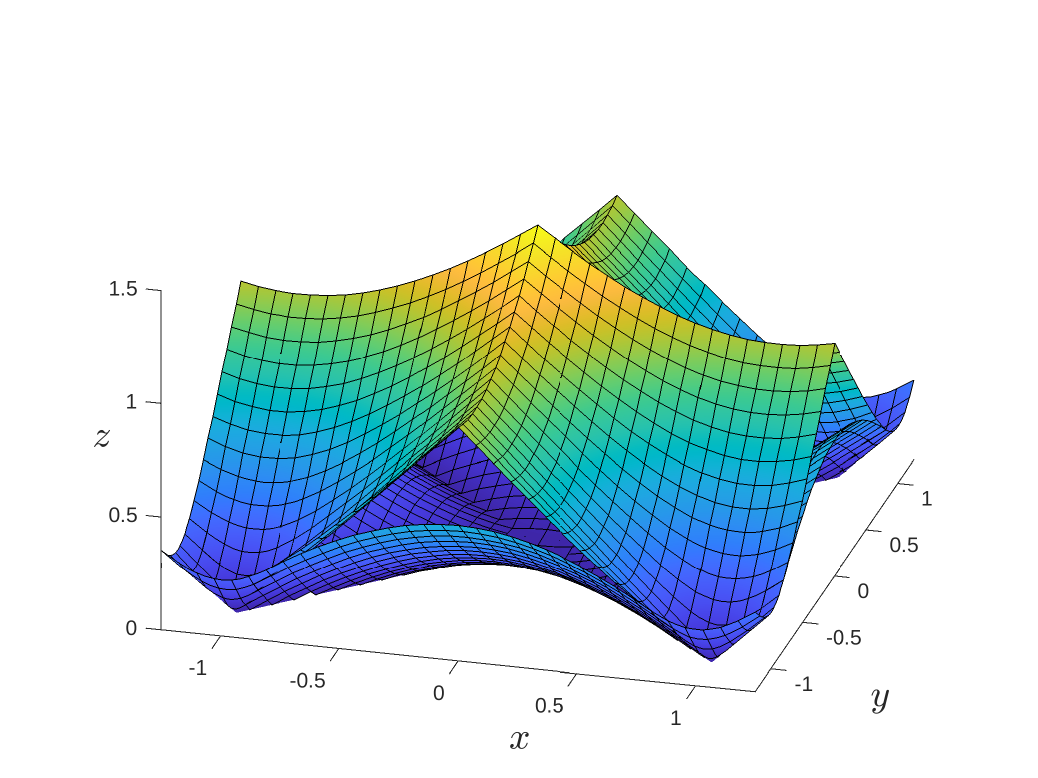}
\par\end{centering}
\caption{The Clifford and quadratic pseudospectrum for the universal pair of order-two unitary matrices $(U,V)$ are plotted together here.}
\end{figure}

\begin{thm}
For the universal pair of unitary operators of order two, so $U$ and $V$ as in \textup{(\ref{eqn:the_two_projectinons})} in the universal $C^*$-algebra \textup{(\ref{eqn:universal_for_two_proj})}, the quadratic pseudospectrum is 
\begin{equation*}
\mu_{(x,y)}^{\textup{Q}}(U,V)=\textup{dist}\left((x,y),\left\{ (-1,-1),(-1,1),(1,-1),(1,1)\strut\right\} \right)
\end{equation*}
and the Clifford pseudospectrum is
\begin{equation}
\mu_{(x,y)}^{\textup{C}}(U,V)=\sqrt{x^{2}+y^{2}+2-2\sqrt{x^{2}y^{2}+x^{2}+y^{2}+1}}\label{eq:CliffordPS_of_U_Vz}
\end{equation}
when \textup{$-1\leq xy\leq1$,} and otherwise
\begin{equation*}
\mu_{(x,y)}^{\textup{C}}(U,V)=\mu_{(x,y)}^{\textup{Q}}(U,V).
\end{equation*}
In particular, the quadratic spectrum is the set of the four points $(\pm1,\pm1)$ and the Clifford spetrum is a cross, the union of the line segment from $(-1,-1)$  to $(1,1)$ and the line segment from $(-1,1)$  to $(1,-1)$.
\end{thm}

\begin{proof}
First the quadratic case. For each $(x,y)$ we need to minimize 
\begin{equation*}
\mu_{(x,y)}^{\textup{Q}}(U,V_{z})=\sqrt{x^{2}+y^{2}+2-2\sqrt{x^{2}+2zxy+y^{2}}}
\end{equation*}
over the range $-1\leq z\leq1$. This is constant when $x=0$ or $y=0$
and in the remaining cases, is maximized when $z=-1$ or when $z=1$.
Since
\begin{equation*}
x^{2}\pm2xy+y^{2}=(x\pm y)^{2}
\end{equation*}
we find that
\begin{equation*}
\mu_{(x,y)}^{\textup{Q}}(U,V_{1})=\sqrt{x^{2}+y^{2}+2-2\sqrt{x^{2}+2xy+y^{2}}}=\begin{cases}
\sqrt{\left(x+1\right)^{2}+\left(y+1\right)^{2}} & \text{if }x+y\leq0\\
\sqrt{\left(x-1\right)^{2}+\left(y-1\right)^{2}} & \text{if }x+y\geq0
\end{cases}.
\end{equation*}
This means $\mu_{(x,y)}^{\textup{Q}}(U,V_{1})$ is the distance of
$(x,y)$ to the set $\{(-1,-1),(1,1)\}$. Similarly $\mu_{(x,y)}^{\textup{Q}}(U,V_{-1})$
is the distance of $(x,y)$ to the set $\{(1,-1),(-1,1)\}$. 

Now we look to the Clifford case. For each $(x,y)$ we need to minimize
\begin{equation}
\mu_{(x,y)}^{\textup{C}}(U,V_{z})=\sqrt{x^{2}+y^{2}+2-2\sqrt{x^{2}+2xyz+y^{2}+1-z^{2}}}.\label{eq:CliffordPS_of_universal_unitaries}
\end{equation}
This will happen at the maximum of
\begin{equation*}
g(z)=x^{2}+2xyz+y^{2}+1-z^{2}.
\end{equation*}
This has derivative
\begin{equation*}
g'(z)=2xy-2z
\end{equation*}
so the maximum value of $g$ seems like it will be at $z=z_{0}$ where $z_{0}=xy$
but we need to ensure this is with $-1\leq z\leq1$. If $(x,y)$ is
in the region specified by $-1\leq xy\leq1$, 
a region that includes $[-1,1]\times[-1,1]$, then the value of $\mu_{(x,y)}^{\textup{C}}(U,V_{z})$
is found at $z_{0}$. Substituting $z_{0}$ for $z$ in Eq.~\ref{eq:CliffordPS_of_universal_unitaries}
yields Eq.~\ref{eq:CliffordPS_of_U_Vz}.

\end{proof}

\section{A hemisphere as Clifford spectrum \label{sec:hemisphere}}

Example~\ref{ex:any_space_occurs} tells us we can find essentially any
finite-dimensional compact space as the Clifford spectrum of some set of commuting Hermitian operators.  In the finite-dimensional case we need to introduce noncommutativity to get Clifford spectrum in dimension greater that zero.  Many papers have examined the spaces that one can get from finite matrices, including \cite{berenstein2012matrix_embeddings,CerjanLoring2023EvenSpheres,DeBonisLorSver_joint_spectrum,sykora2016fuzzy_space_kit}.  With three matrices it is possible to have as Clifford spectrum a $2$-manifold or a surface with cusps.  With four matrices one can get $2$-manifolds and $3$-manifold, for example.
The simplest example leading to a closed surface is the Pauli matrices themselves \cite{Kisil_monogenic_func_calc}, where 
\begin{equation*}
    \Lambda(\sigma_x,\sigma_y,\sigma_z) = S^2.
\end{equation*}
We find some odd behavior when we modify this example to be infinite-dimensional.

Consider
\begin{equation}
\label{eqn:hemisphere_example}
A_{1}=\frac{1}{2}\left[\begin{array}{ccccc}
0 & 1\\
1 & 0 & 1\\
 & 1 & 0 & 1\\
 &  & 1 & \ddots & \ddots\\
 &  &  & \ddots & \ddots
\end{array}\right],\ A_{2}=\frac{1}{2}\left[\begin{array}{ccccc}
0 & -i\\
i & 0 & -i\\
 & i & 0 & -i\\
 &  & i & \ddots & \ddots\\
 &  &  & \ddots & \ddots
\end{array}\right],\ A_{3}=\left[\begin{array}{ccccc}
b\\
 & 0\\
 &  & 0\\
 &  &  & \ddots\\
 &  &  &  & \ddots
\end{array}\right].
\end{equation}
If we set $b=1/2$ we have operators that are somewhat like 
$\tfrac{1}{2}\sigma_x,\tfrac{1}{2}\sigma_y,\tfrac{1}{2}\sigma_z$.  We will find that for this class of examples, we get a hemisphere and not a sphere.  This is in some way a manifestation that the $K$-theory of $\mathcal{B}(\mathcal{H})$ is trivial so there can be no local index stabilizing the Clifford spectrum.

Notice that $A_1$, $A_2$, $A_3$ are compact perturbations of 
\begin{equation*}
    K_1=\frac{1}{2}(S^*+S), K_2=\frac{i}{2}(S^*-S), K_3=0
\end{equation*} 
where $S$ is the usual shift operator.  Thus the essential Clifford spectrum in this example will not depend on $b$. 

\begin{lem}
For any  $b\geq 0$, if $H_{1}$, $H_{2}$, $H_{3}$ are defined
according to Equation~\textup{\ref{eqn:hemisphere_example}}, then
\begin{equation*}
\Lambda^{\textup{e}}(H_{1},H_{2},H_{3})=\mathbb{T}\times\{0\}.
\end{equation*}
\end{lem}

\begin{proof}
Since the essential spectrum of the unilateral shift is the unit circle we know that
\begin{equation*}
\Lambda^{\textup{e}}(H_{1},H_{2})=\mathbb{T}.
\end{equation*}
Now Theorem~\ref{thm:last_is_scalar} implies 
\begin{equation*}
\Lambda^{\textup{e}}(H_{1},H_{2},H_{3})=\Lambda^{\textup{e}}(H_{1},H_{2})\times\{0\}
\end{equation*}
so we are done.
\end{proof}

\begin{example}
\label{ex:Three_Toeplitz_type}
For any  $0 \leq b\leq 2.25$, if $H_{1}$, $H_{2}$, $H_{3}$ are defined
according to Equation~\ref{eqn:hemisphere_example}, then
\begin{equation*}
\Lambda(H_{1},H_{2},H_{3}).
\end{equation*}
equals the set of points $(x,y,z)$ that satisfy the equation
\begin{equation*}
    b^{4}z+b^{3}r^{2}-3b^{3}z^{2}-b^{3}-b^{2}r^{2}z+3b^{2}z^{3}+2b^{2}z+br^{4}-br^{2}-bz^{4}-bz^{2}+r^{2}z=0.
\end{equation*}
where $r^2= x^2 + y^2$
and subject to the constraints $0\leq z\leq b$ and $x^2 + y^2 \leq 1$.
\end{example}

We will verify this with a rigorous proof in the special case $b=1$.
When $b=0$ this follows easily from Theorem~\ref{thm:last_is_scalar}.  For the general case, we use an experimental mathematics approach, where we use computer algebra to derive an equality and an inequality that together determine part of the Clifford spectrum and then estimate the joint solution numerically.  

For any $b\geq 0$
 it is easy to see that $e^{i\theta}S$ is unitarily equivalent to $S$.  This implies that\begin{equation*}
   \left(\cos(\theta)H_1+\sin(\theta)H_2,\ 
   \cos(\theta)H_1-\sin(\theta)H_2,\ 
   H_3 \right)
\end{equation*}     
is unitarily equivalent to $H_1,H_2,H_3$.  We can now apply
Theorem~\ref{thm:symmetry_in_spectrum} to conclude that 
$\Lambda(H_{1},H_{2},H_{3})$ has rotational symmetry in the $x$-$y$-plane.  Thus, we can restrict our attention to the localizer with $\boldsymbol{\lambda}=(x,0,z)$ for $x\geq 0$.

The localizer naturally lives in $\mathbf{M}_2(\mathcal{B}(\mathcal{H}))$ but we can identify that with $\mathcal{B}(\mathcal{H})$ via the usual shuffling of basis elements.  Making this identification we can compute $L_{(-x,0,z)}$ as follows. Let $E_{ij}$ denote the usual two-by-two matrix units, so that $\sigma_y = -iE_{12}+iE_{21}$ etc. Since
\begin{align*}
A_{1}\otimes\sigma_{x}+A_{2}\otimes\sigma_{y} & =A_{1}\otimes(E_{12}+E_{21})+iA_{2}\otimes(-E_{12}+E_{21})\\
 & =(A_{1}-iA_{2})\otimes E_{12}+(A_{1}+iA_{2})\otimes E_{21}\\
 & =S^{*}\otimes E_{12}+S\otimes E_{21}
\end{align*}
we find
\begin{equation*}
A_{1}\otimes\sigma_{x}+A_{2}\otimes\sigma_{y}=\left[\begin{array}{ccccccc}
0 & 0 &  &  &  &  &  \\
0 & 0 & 1 &  &  &  &  \\
 & 1 & 0 & 0 &  &  &  \\
 &  & 0 & 0 & 1 &  &  \\
 &  &  & 1 & 0 & 0 &  \\
 &  &  &  & 0 & 0 & \ddots\\
  &   &   &   &   & \ddots & \ddots
\end{array}\right] .  
\end{equation*}
Therefore
\begin{equation*}
L_{(x,0,z)}  =\left(A_{1}-xI\right)\otimes\sigma_{x}+A_{2}\otimes\sigma_{y}+\left(A_{3}-zI\right)\otimes\sigma_{z}
\end{equation*}
works out as
\begin{equation*}
L_{(x,0,z)} = 
\left[\begin{array}{ccccccc}
b-z & -x &   &   &   &   &  \\
-x & z-b & 1 &  &  &  &  \\
 & 1 & -z & -x &  &  & \\
 &  & -x & z & 1 &  & \\
 &  &  & 1 & -z & -x &  \\
  &   & &  & -x & z & \ddots\\
  &   &  &  &   & \ddots & \ddots
\end{array}\right].   
\end{equation*}

We we already worked out the essential Clifford spectrum so we already know that $0$ is in the essential spectrum of the spectral localizer.  What is left is find and discrete spectrum.  Thus, we search for values of $x$ and $z$ that  lead to null vectors for $L_{(-x,0,z)}$.

\begin{figure}
\includegraphics[scale=0.75]{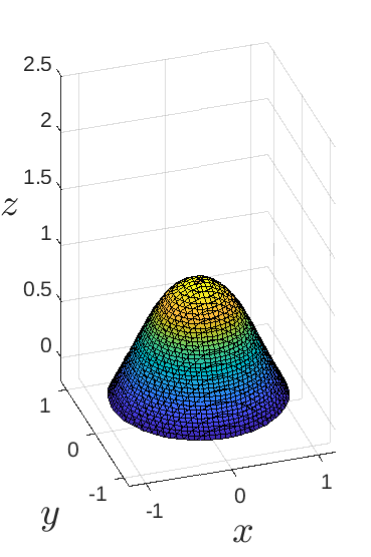}\hspace*{-.1cm}
\includegraphics[scale=0.75]{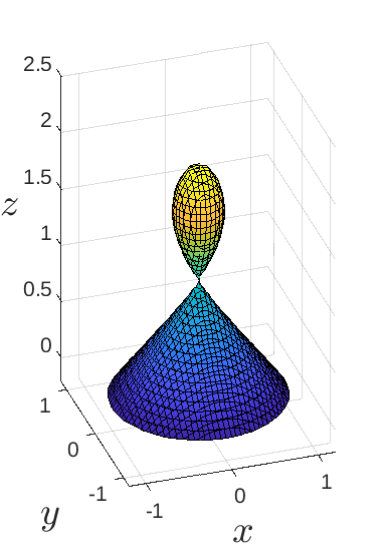}\hspace*{-0.1cm}
\includegraphics[scale=0.75]{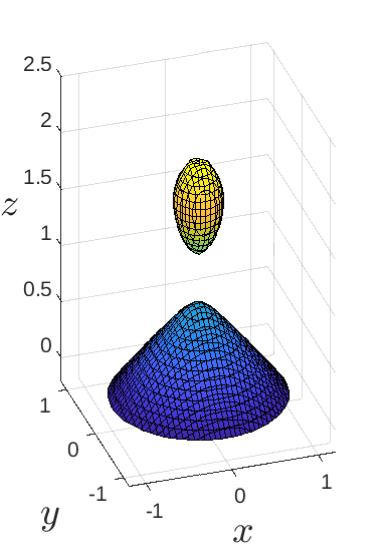}
\caption{The Clifford spectrum for the three operators in Equation~\ref{eqn:hemisphere_example}.  $b=1.00$ (left);  $b=2.00$ (center);  $b=2.05$ (right); }
\end{figure}

The calculations in the general case are complicated so now it is time to restrict to the case of $b=1$.

We need to discuss several special cases. 
First we deal with $x=0$. This makes $L_{(x,0,z)}$ block diagonal,
with blocks
\begin{equation*}
\left[1-z\right],\ \left[\begin{array}{cc}
z-1 & 1\\
1 & -z
\end{array}\right],\ \left[\begin{array}{cc}
z & 1\\
1 & -z
\end{array}\right]
\end{equation*}
that have determinants
$1-z$ $-z^{2}+z-1$ and $-z^{2}-1$,
so only $z=1$ leads to $L_{(x,0,z)}$ being singular. 

Now assume $x\geq 0$.
Let us assume that there exists a nonzero vector $\boldsymbol{a}=(a_{1},a_{2},\dots)\in l^{2}(\mathbb{N})$
so that $L_{(x,0,z)}\boldsymbol{a}=0$. We can rescale $\boldsymbol{a}$
so that $a_{1}=1$. This means
\begin{align*}
(1-z)-xa_{2} & =0,\\
-x+(z-1)a_{2}+a_{3} & =0,
\end{align*}
 and for $n\ge2$, 
\begin{align*}
a_{2n-2}-za_{2n-1}-xa_{2n} & =0\\
-xa_{2n-1}+za_{2n}+a_{2n+1} & =0.
\end{align*}
Thus, for $n\geq2$,
\[
a_{2n}=-\frac{z}{x}a_{2n-1}+\frac{1}{x}a_{2n-2}
\]
and
\begin{align*}
a_{2n+1} & =-za_{2n}+xa_{2n-1}\\
 & =-z\left(-\frac{z}{x}a_{2n-1}+\frac{1}{x}a_{2n-2}\right)+xa_{2n-1}\\
 & =\left(\frac{z^{2}}{x}+x\right)a_{2n-1}-\frac{z}{x}a_{2n-2}.
\end{align*}
If we set 
\begin{equation*}
M=\left[\begin{array}{cc}
\frac{1}{x} & -\frac{z}{x}\\
-\frac{z}{x} & \frac{x^{2}+z^{2}}{x}
\end{array}\right]
\end{equation*}
the relation on $\boldsymbol{a}$ is
\begin{equation*}
\left[\begin{array}{c}
a_{2n}\\
a_{2n+1}
\end{array}\right]=M\left[\begin{array}{c}
a_{2n-2}\\
a_{2n-1}
\end{array}\right] .
\end{equation*}
We also know
\begin{equation*}
a_{2}=\frac{1-z}{x}
\end{equation*}
and 
\begin{align*}
a_{3} & =x+(1-z)a_{2}\\
 & =x+(1-z)\frac{1-z}{x}
\end{align*}
so have an initial condition
\begin{equation*}
\boldsymbol{v}_{0}=\left[\begin{array}{c}
a_{2}\\
a_{3}
\end{array}\right]=\left[\begin{array}{c}
\frac{1-z}{x}\\
x+\frac{\left(1-z\right)^{2}}{x}
\end{array}\right].
\end{equation*}

Now let's pause to take care of the case of $z=1$ and $x>0$. In
that situation, we have
\begin{equation*}
M=\left[\begin{array}{cc}
\frac{1}{x} & -\frac{1}{x}\\
-\frac{1}{x} & \frac{x^{2}+1}{x}
\end{array}\right]
\end{equation*}
so that 
\begin{equation*}
\boldsymbol{v}_{0}=\left[\begin{array}{c}
0\\
x
\end{array}\right]
\end{equation*}
is not going to be an eigenvalue. Thus, the only Clifford spectrum
on the line $z=1$ is the point $(x,0,z)=(0,0,1)$ that we found above.

We need to know about the eigenvalues of the real symmetric matrix
$M$. a see that
\begin{equation*}
\textup{Tr}(M)=\frac{x^{2}+z^{2}+1}{x}
\end{equation*}
and $\det(M)=1$ we find that the eigenvalues of the real symmetric
matrix $M$ are
\begin{equation*}
\frac{1}{2}\left(\textup{Tr}(M)\pm\sqrt{\left(\textup{Tr}(M)\right)^{2}-4}\right).
\end{equation*}
We get a double real eigenvalue when 
\begin{equation*}
\frac{x^{2}+z^{2}+1}{x}=2
\end{equation*}
but this leads to $x=1$ and $z=0$ which we have excluded. We get
two complex eigenvalues when 
\begin{equation*}
\frac{x^{2}+z^{2}+1}{x}<2
\end{equation*}
which is equivalent to 
\begin{equation*}
\left(x-1\right)^{2}+z^{2}<0
\end{equation*}
and that cannot happen. What we need look at then is where $M$ has
two positive eigenvalues, one outside of $(0,1)$ and one inside.
The only way to have $\boldsymbol{a}$ square-summable is if the initial
vector lands in the eigenspace for the eigenvalue closer to $0$.

Let us find out when $M\boldsymbol{v}_{0}$ is parallel to $\boldsymbol{v}_{0}$.
Since
\begin{equation*}
M\boldsymbol{v}_{0} =
\frac{1}{x^{2}}\left[\begin{array}{c}
\left(1-z\right)-z\left(x^{2}+z^{2}-2z+1\right)\\
-z\left(1-z\right)+\left(x^{2}+z^{2}\right)\left(x^{2}+z^{2}-2z+1\right)
\end{array}\right]
\end{equation*}
we need to solve
\begin{equation*}
\frac{x^{-2}\left(\left(1-z\right)-z\left(x^{2}+z^{2}-2z+1\right)\right)}{x^{-1}\left(1-z\right)}-\frac{x^{-2}\left(-z\left(1-z\right)+\left(x^{2}+z^{2}\right)\left(x^{2}+z^{2}-2z+1\right)\right)}{x^{-1}\left(x^{2}+z^{2}-2z+1\right)}=0.
\end{equation*}
For $x\neq0$ and $z\neq1$ this is equivalent to 
\begin{equation*}
-\frac{x^{4}-z^{4}+3z^{3}-4z^{2}+3z-1}{x^{2}+z^{2}-2z+1}=0
\end{equation*}
so we get the curve 
$e(x,z)=0$  where
\begin{equation*}
e(x,z)=x^{4}-z^{4}+3z^{3}-4z^{2}+3z-1.
\end{equation*}

\begin{figure}
\includegraphics[scale=0.75]{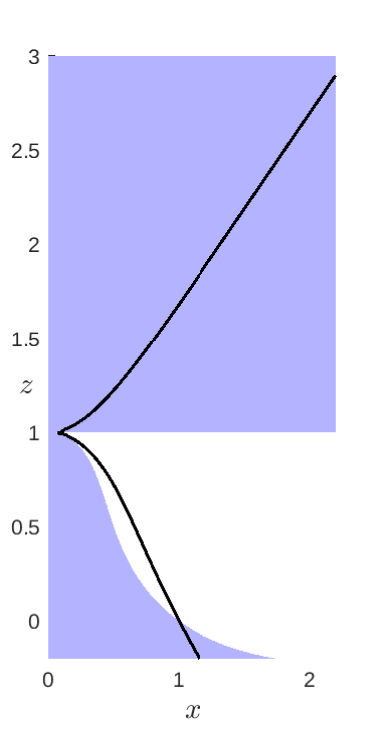}\hspace*{-.1cm}
\includegraphics[scale=0.75]{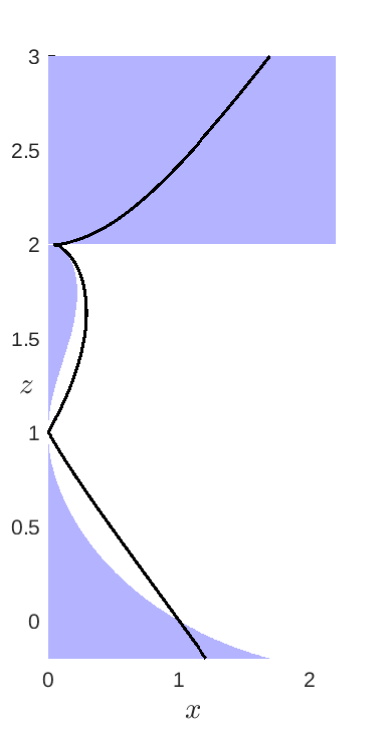}\hspace*{-.1cm}
\includegraphics[scale=0.75]{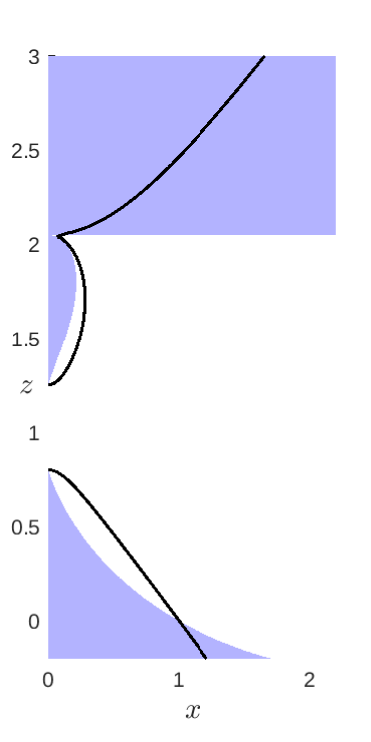}
\caption{Plots of $e(x,z)=0$ (dotted curves) and  $f(x,z)\geq0 $ (light blue area, switched to $\leq$ above $z=b$) for $b=1.00$, $b=2.00$, $b=2.05$.  \label{fig:b_one_plots}}
\end{figure}

What about the associated eigenvalue $\alpha$? Is it between $0$
and $1$? First, we deal with the case where $z<1$. Here the first
coordinate of the initial vector $\boldsymbol{v}_{0}$ will be positive
so  $\alpha$ will be positive when 
\begin{equation*}
\frac{1-z}{x}>\frac{-z\left(1-z\right)+\left(x^{2}+z^{2}\right)\left(x^{2}+z^{2}-2z+1\right)}{x^{2}}.
\end{equation*}
As $x^{2}$ is positive this is equivalent to
\begin{equation*}
x-xz>-x^{2}z-z^{3}+2z^{2}-2z+1
\end{equation*}
so we need to know when
\begin{equation*}
f(x,z)=x^{2}z-xz+x+z^{3}-2z^{2}+2z-1
\end{equation*}
is positive. For $z>1$ we need to know when $f(x,z)$ is negative.

We find
\begin{equation*}
\frac{\partial f}{\partial z}=x^{2}-x+3z^{2}-4z+2=(x-\tfrac{1}{2})^{2}+3(z-\tfrac{2}{3})^{2}+\tfrac{5}{12}
\end{equation*}
is always positive. The plots of the zero-locus of $e$ and $f$ are
shown in Figure~\ref{fig:b_one_plots}. These seem to cross at $(x,z)=(0,1)$ and $(x,z)=(1,0)$
and the curve for $f=0$ is below the curve for $e=0$ between those
two crossings. If we believe the computer-generated plots then we can conclude that the discrete Clifford spectrum is
on the curve $e(x,z)=0$ restricted to the region $0\leq x,z\leq1$. 

We can verify the above claim rigorously, with some uninteresting calculus.  We put the rest of the proof in Appendix~\ref{sec:Finish_Example_b_1}. 

When $b\neq 1$ the polynomials $e(x,z)$  and $f(x,z)$  are trickier.  As such, we resorted to using symbolic computer algebra for various values of $b$ and numerically plot the curve $e(x,z)=0$ and regions associated to $f(x,z)$. While we do not have a fully rigorous proof, we feel this is sufficient to understand this example.

\section{Momentum and Position -- an unbounded example}

We hope to push the idea of joint pseudospectrum into the realm of unbounded operators. Here we
consider a classic example, and do not attempt a general theory.   The example is a classic interpretation of position and momentum.

We will work on the Hilbert space $L^{2}(\mathbb{R})$, and in particular
on the subspace $\mathcal{S}$ of Schwartz functions, where we conisder only $f$ for which  $x^{m}D^{n}f(x)$
is always bounded.  Here we use
\begin{equation*}
D(f)(x)=\frac{d}{dx}(f(x))
\end{equation*}
to denote the standard differential operator. 
We consider two unbounded operators, $P,Q:\mathcal{S}\rightarrow\mathcal{S}$
where $P=-i D$ and $Q(f)(x)=xf(x).$
These are both symmetric (but probably not self-adjoint, so $Q\subseteq Q^{*}$
is all that is claimed, due to the small domain).

We can consider
\begin{equation*}
Q_{\boldsymbol{\lambda}}(P,Q)
=\left( P-\lambda_{1}I\right)^2
+\left( Q-\lambda_{2}I\right)^2
:\mathcal{S}\rightarrow\mathcal{S}.
\end{equation*}
and define
\begin{equation*}
\mu_{\boldsymbol{\lambda}}^{\textup{Q}}(P,Q)=
\sqrt{\min_{\|f\|=1}\|Q_{\boldsymbol{\lambda}}(P,Q)(f)\|}.
\end{equation*}
The Clifford pseudospectrum is also easy to define,
as 
\begin{equation*}
\mu_{\boldsymbol{\lambda}}^{\textup{C}}(P,Q)=\min_{\|f\|=1}\|L_{\boldsymbol{\lambda}}(P,Q)(f)\|
\end{equation*}
where
\begin{equation*}
L_{\boldsymbol{\lambda}}(P,Q)=\left[\begin{array}{cc}
0 & \left(P-\lambda_{1}I\right)-i\left(Q-\lambda_{2}I\right)\\
\left(P-\lambda_{1}I\right)+i\left(Q-\lambda_{2}I\right) & 0
\end{array}\right]
\end{equation*}
is taken to be an operator on ${S}\oplus\mathcal{S}$.
In \cite{DeBonisLorSver_joint_spectrum}, this example of Clifford pseudospectrum was computed, with
\begin{equation*}
\mu_{\boldsymbol{\lambda}}^\textup{C}(P,Q)\equiv 0
\end{equation*}
being the result. 

We will need some alternate characterizations of the quadratic pseudospectrum of $P$ and $Q$, as proven in \cite{cerjan_L_V_2022quadraticPS} for the matrix case.
We can consider
\begin{equation*}
M_{\boldsymbol{\lambda}}(P,Q)=\left[\begin{array}{c}
P-\lambda_{1}I\\
Q-\lambda_{2}I
\end{array}\right]:\mathcal{S}\rightarrow\mathcal{S}\oplus\mathcal{S}.
\end{equation*}
and can check that
\begin{align*}
\mu_{\boldsymbol{\lambda}}^{Q}(P,Q) & =\sqrt{\min_{\|f\|=1}\left\Vert \left(M_{\boldsymbol{\lambda}}(P,Q)(f)\right)^{*}\left(M_{\boldsymbol{\lambda}}(P,Q)(f)\right)\right\Vert }\\
 & =\min_{\|f\|=1}\|M_{\boldsymbol{\lambda}}(P,Q)(f)\|\\
 & =\min_{\|f\|=1}\sqrt{\left\Vert Pf-\lambda_{1}f\right\Vert ^{2}+\left\Vert Qf-\lambda_{2}f\right\Vert ^{2}}.
\end{align*}

Since 
\begin{equation*}
\left\Vert A\boldsymbol{v}-\lambda\boldsymbol{v}\right\Vert ^{2}=\Delta_{\boldsymbol{v}}^{2}A+\left(\textup{E}_{\boldsymbol{v}}(A)-\lambda\right)^{2}
\end{equation*}
for any Hermitian operator, we can find lower bounds on the quadratic
pseudospectrum by utilizing results on lower bounds on the sum of
uncertainty. For example, there are lower bounds on $\Delta_{\boldsymbol{v}}^{2}A+\Delta_{\boldsymbol{v}}^{2}B$
in \cite{maccone_Pati2014StrongerUncertainty}. For simplicity, we give a simplification of a proof from
\cite{maccone_Pati2014StrongerUncertainty} and directly prove the following.

\begin{lem}
\label{lem:Simple_lower_bound_for_QPS} If $\boldsymbol{v}$ is a unit
vector and $A$ and $B$ are Hermitian operators then 
\begin{equation*}
\left\Vert A\boldsymbol{v}\right\Vert ^{2}+\left\Vert B\boldsymbol{v}\right\Vert ^{2}\geq\left|\left\langle i\left[A,B\right]\boldsymbol{v},\boldsymbol{v}\right\rangle \right|.
\end{equation*}
\end{lem}

\begin{proof}
We  first calculate
\begin{equation*}
\left\Vert \left(A\mp iB\right)\boldsymbol{v}\right\Vert ^{2}=\left\Vert A\boldsymbol{v}\right\Vert ^{2}+\left\Vert B\boldsymbol{v}\right\Vert ^{2}\mp\left\langle i\left[A,B\right]\boldsymbol{v},\boldsymbol{v}\right\rangle 
\end{equation*}
so we have 
\begin{equation*}
\left\Vert A\boldsymbol{v}\right\Vert ^{2}+\left\Vert B\boldsymbol{v}\right\Vert ^{2}=\left\Vert \left(A\mp iB\right)\boldsymbol{v}\right\Vert ^{2}\pm\left\langle i\left[A,B\right]\boldsymbol{v},\boldsymbol{v}\right\rangle .
\end{equation*}
An easy consequence is then 
\begin{equation*}
\left\Vert A\boldsymbol{v}\right\Vert ^{2}+\left\Vert B\boldsymbol{v}\right\Vert ^{2}\geq\pm\left\langle i\left[A,B\right]\boldsymbol{v},\boldsymbol{v}\right\rangle .
\end{equation*}
\end{proof}

\begin{thm}
For $P$ and $Q$ as above, 
\begin{equation*}
\mu_{\boldsymbol{\lambda}}^{\textup{Q}}(P,Q)=1
\end{equation*}
for all $\boldsymbol{\lambda}$.
\end{thm}

\begin{proof}
We start by establishing that $\mu_{\boldsymbol{\lambda}}^\textup{Q}(P,Q)$
is constant. The shift operator $f(x)\mapsto f(x-\lambda_{2})$ gives a
unitary $U$ that takes $\mathcal{S}$ onto $\mathcal{S}$.  Since
\begin{align*}
U^{*}PU(f(x)) 
 & =-i U^{*}(f'(x-\lambda_{2}))\\
 & =-i f'(x)\\
 & =P(f(x))
\end{align*}
and
\begin{align*}
U^{*}QU(f(x)) & =U^{*}(xf(x-\lambda_{2}))\\
 & =(x+\lambda_{2})f(x)
\end{align*}
we find
$U^{*}PU=P$ and $U^{*}QU=Q+\lambda_{2}$.
Therefore $Q_{(\lambda_1,\lambda_2)}(P,Q)$
is unitarily equivalent to
$Q_{(\lambda_1,0)}(P,Q)$ and so
\begin{equation*}
\mu_{(\lambda_{1},\lambda_{2})}^\textup{Q}(P,Q)=\mu_{(\lambda_{1},0)}^\textup{Q}(P,Q).
\end{equation*}
Likewise, we can use $f(x)\mapsto e^{2\pi i\lambda_1 x}f(x)$ so set up
a unitary that shows 
\begin{equation*}
\mu_{(\lambda_{1},\lambda_{2})}^\textup{Q}(P,Q)=\mu_{(0,\lambda_{2})}^\textup{Q}(P,Q).
\end{equation*}
Taken together, these facts prove that the quadratic 
pseudospectrum is constant.

Since $\left[P,Q\right](f)=-i f$,
 Lemma~\ref{lem:Simple_lower_bound_for_QPS} says 
\begin{equation*}
\left\Vert Pf\right\Vert ^{2}+\left\Vert Qf\right\Vert ^{2}\geq 1.
\end{equation*}
Therefore $\mu_{(\boldsymbol{0})}^{\textup{Q}}(P,Q)\geq 1$ and so 
\begin{equation*}
\mu_{\boldsymbol{\lambda}}^{\textup{Q}}(P,Q)\geq 1
\end{equation*}

Now consider
$g(x)=e^{-\frac{1}{2}x^{2}}$.
Clearly 
\begin{equation*}
Q^{2}(g)(x)=x^{2}e^{-\frac{1}{2}x^{2}}=x^{2}g(x).
\end{equation*}
We also find 
\begin{align*}
P^{2}(g)(x) & =P\left(ixe^{-\frac{1}{2}x^{2}}\right)\\
 & =-i\frac{d}{dx}\left(ixe^{-\frac{1}{2}x^{2}}\right)\\
 & =\frac{d}{dx}\left(xe^{-\frac{1}{2}x^{2}}\right)\\
 & = e^{-\frac{1}{2}x^{2}}-x^{2}e^{-\frac{1}{2}x^{2}}\\
 & = g(x)-x^{2}g(x)
\end{align*}
and so find
\begin{equation*}
\left(P^{2}+Q^{2}\right)g= g.
\end{equation*}
Having found an eigenvector $Q_{(0,0)}(P,Q)=P^{2}+Q^{2}$ we have proven
\begin{equation*}
\mu_{\boldsymbol{\lambda}}^{\textup{Q}}(P,Q) \leq 1.
\end{equation*}
\end{proof}

\section*{Acknowledgments}
A.C.\ and V.L.\ acknowledge support from the Laboratory Directed Research and Development program at Sandia National Laboratories. 
T.A.L.\ acknowledges support from the National Science Foundation, Grant No. DMS-2110398. 
This work was performed in part at the Center for Integrated Nanotechnologies, an Office of Science User Facility operated for the U.S. Department of Energy (DOE) Office of Science.
Sandia National Laboratories is a multimission laboratory managed and operated by National Technology \& Engineering Solutions of Sandia, LLC, a wholly owned subsidiary of Honeywell International, Inc., for the U.S. DOE's National Nuclear Security Administration under Contract No. DE-NA-0003525. 
The views expressed in the article do not necessarily represent the views of the U.S. DOE or the United States Government.

\appendix

\newpage

\section{Calculations for Example~\ref{ex:Three_Toeplitz_type}}

\label{sec:Finish_Example_b_1}

Here if the rest of the proof for the case $b=1$.  We can solve for $x$ in terms of $z$ on the $e=0$ curve, so
\begin{equation*}
x=\left(z^{4}-3z^{3}+4z^{2}-3z+1\right)^{1/4}.
\end{equation*}
which simplifies to 
\begin{equation*}
x=\left|z-1\right|^{\frac{1}{2}}\left(z^{2}-z+1\right)^{1/4}.
\end{equation*}
This substitution into $f(x,z)$ where,
\begin{equation*}
f(x,z)=x^{2}z-x(z-1)+\left(z-1\right)\left(z^{2}-z+1\right),
\end{equation*}
leads to the one-variable function 
\begin{equation*}
f_{e}(z)=z\left|z-1\right|\left(z^{2}-z+1\right)^{1/2}-(z-1)\left|z-1\right|^{\frac{1}{2}}\left(z^{2}-z+1\right)^{1/4}+\left(z-1\right)\left(z^{2}-z+1\right)
\end{equation*}

For $z\leq1$ this becomes
\begin{equation*}
f_{e}(z)=z\left(1-z\right)\left(z^{2}-z+1\right)^{1/2}+(1-z)\left(1-z\right)^{\frac{1}{2}}\left(z^{2}-z+1\right)^{1/4}-\left(1-z\right)\left(z^{2}-z+1\right).
\end{equation*}
For $0\leq z\leq1$ we have $z^{2}-z+1\leq1$ so
\begin{align*}
f_{e}(z) & \geq z\left(1-z\right)\left(z^{2}-z+1\right)^{1/2}+(1-z)\left(1-z\right)^{\frac{1}{2}}\left(z^{2}-z+1\right)^{1/2}-\left(1-z\right)\left(z^{2}-z+1\right)^{\frac{1}{2}}\\
 & =\left(z(1-z)+(1-z)(1-z)^{\frac{1}{2}}-\left(1-z\right)\right)\left(z^{2}-z+1\right)^{\frac{1}{2}}\\
 & =(1-z)\left(z+(1-z)^{\frac{1}{2}}-(1-z)\right)\left(z^{2}-z+1\right)^{\frac{1}{2}}\\
 & =(1-z)\left(2z-1+(1-z)^{\frac{1}{2}}\right)\left(z^{2}-z+1\right)^{\frac{1}{2}}\\
 & \geq0.
\end{align*}
For $z\leq0$ we have $z^{2}-z+1\geq1$ so
\begin{align*}
f_{e}(z) & \leq z\left(1-z\right)\left(z^{2}-z+1\right)^{1/4}+(1-z)\left(1-z\right)^{\frac{1}{2}}\left(z^{2}-z+1\right)^{1/4}-\left(1-z\right)\left(z^{2}-z+1\right)\\
 & =(1-z)\left(z\left(z^{2}-z+1\right)^{1/4}+\left(1-z\right)^{\frac{1}{2}}\left(z^{2}-z+1\right)^{1/4}-\left(z^{2}-z+1\right)\right)\\
 & =(1-z)\left(z^{2}-z+1\right)^{1/4}\left(z+\left(1-z\right)^{\frac{1}{2}}-\left(z^{2}-z+1\right)^{\frac{3}{4}}\right)
\end{align*}
One can see this is at most zero by looking at its derivative.

For $z\geq1$ the formula for $f_{e}$ becomes
\begin{equation*}
f_{e}(z)=z\left(z-1\right)\left(z^{2}-z+1\right)^{1/2}-(z-1)\left(z-1\right)^{\frac{1}{2}}\left(z^{2}-z+1\right)^{1/4}+\left(z-1\right)\left(z^{2}-z+1\right).
\end{equation*}
Here $z^{2}-z+1\geq1$ so
\begin{align*}
f_{e}(z) & =\left(z-1\right)\left(z\left(z^{2}-z+1\right)^{1/2}-\left(z-1\right)^{\frac{1}{2}}\left(z^{2}-z+1\right)^{1/4}+\left(z^{2}-z+1\right)\right)\\
 & \geq\left(z-1\right)\left(z\left(z^{2}-z+1\right)^{1/2}-\left(z-1\right)^{\frac{1}{2}}\left(z^{2}-z+1\right)^{1/2}+\left(z^{2}-z+1\right)^{\frac{1}{2}}\right)\\
 & =\left(z-1\right)\left(z^{2}-z+1\right)^{1/2}\left(z+1-\left(z-1\right)^{\frac{1}{2}}\right)\\
 & \geq0.
\end{align*}

\end{document}